\documentclass{article}

\usepackage[a4paper]{geometry}

\usepackage{bm}

\usepackage{amsmath}
\usepackage{amssymb}
\usepackage{amsthm}

\usepackage{tikz}
\usetikzlibrary{matrix,arrows}

\usepackage[all]{xy}

\usepackage{pb-diagram}
\usepackage{pb-xy}

   {\begin{proof}[Sketch of Proof of #1]}%
   {\end{proof}}

\theoremstyle{plain}
  \newtheorem{theorem}{Theorem}[section]
  \newtheorem{corollary}[theorem]{Corollary}
  
  \newtheorem{lemma}[theorem]{Lemma}
  
  \newtheorem{proposition}[theorem]{Proposition}

\theoremstyle{definition}
  \newtheorem{definition}[theorem]{Definition}

  \newtheorem{ex}[theorem]{Example}

  \newtheorem{remark}[theorem]{Remark}
  
  \newtheorem{question}[theorem]{Question}

  \newenvironment{example}{\begin{ex}}{\qed\end{ex}}

\newcommand{\category}[1]{\mathbf{#1}}

  \newcommand{\Top}{\category{Top}}

\newcommand{\coequalizer}[5]{\xymatrix{
 {\displaystyle #1} 
 \ar@<1ex>[r]^-{#2} \ar@<-1ex>[r]_-{#3}
 & {\displaystyle #4} \ar[r] & {\displaystyle #5}
 }%
}
\newcommand{\equalizer}[5]{
\xymatrix{
 {\displaystyle #1} \ar[r] & {\displaystyle #2}
 \ar@<1ex>[r]^-{#3} \ar@<-1ex>[r]_-{#4}
 & {\displaystyle #5} 
 }%
}
  
  \newcommand{\Ima}{\operatorname{Im}}

\def\polhk#1{\setbox0=\hbox{#1}{\ooalign{\hidewidth
    \lower1.5ex\hbox{`}\hidewidth\crcr\unhbox0}}} 

  \newcommand{\quotient}[2]{%
  \left(#1\right)
  \hspace{-4pt}\raisebox{-5pt}{$\bigg/$}\hspace{-2pt}\raisebox{-12pt}{$#2$}%
  }

\newcommand{\set}[2]{%
\left\{#1 \,\middle|\, #2\right\}
}

  \newcommand{\rarrow}[1]{\buildrel #1 \over \longrightarrow}

  \newcommand{\pr}{\operatorname{\mathrm{pr}}}

  \newcommand{\transpose}[1]{\raisebox{1ex}{$\scriptstyle t$}\kern-0.2ex #1}

  \newcommand{\R}{\mathbb{R}}

  \newcommand{\op}{{\operatorname{\mathrm{op}}}} 

  \newcommand{\Lk}{\mathrm{Lk}} 

  \newcommand{\inj}{\operatorname{\mathrm{inj}}}

  \newcommand{\Exit}{\mathsf{Exit}} 
  \newcommand{\cpt}{\mathrm{cpt}} 
  \newcommand{\Haus}{\mathrm{Haus}} 

 \newcommand{\Sing}{\mathrm{Sing}}

  \newcommand{\String}{\mathrm{\String}}

  \newcommand{\St}{\mathrm{St}}

\newcommand{\fixhyperref}{%
\ifnum 42146=\euc"A4A2 \AtBeginDvi{}\else
\AtBeginDvi{}\fi}

\newcommand{\DDelta}{\underline{\Delta}}
\newcommand{\nN}{\overline{N}}
\DeclareMathOperator{\inte}{Int}
\newcommand{\cone}{\mathsf{cone}}
\newcommand{\SSS}{\category{SSS}}
\newcommand{\CNSSS}{\category{CNSSS}}

\usepackage{hyperref}

\title{\bfseries Stellar Stratifications on Classifying Spaces}
\author{Dai Tamaki and Hiro Lee Tanaka}

\begin{document}

\maketitle

\begin{abstract}
 We extend Bj{\"o}rner's characterization of the face poset of finite CW
 complexes to a certain class of stratified spaces, called cylindrically
 normal stellar complexes. As a direct consequence, we obtain a
 discrete analogue of cell decompositions in smooth Morse theory, by
 using the classifying space model introduced in \cite{1612.08429}.
 As another application, we show that the exit-path category $\Exit(X)$,
 in the sense of \cite{LurieHigherAlgebra},
 of a finite cylindrically normal CW stellar complex $X$ is a
 quasicategory. 
\end{abstract}

\tableofcontents

\section{Introduction}
\label{SSS_intro}

In this paper, we study stratifications on classifying spaces of acyclic
topological categories. In particular, the following three questions are
addressed. 

\begin{question}
 \label{question:recover}
 How can we recover the original category $C$ from its classifying space
 $BC$? 
\end{question}

\begin{question}
 \label{question:exit-path}
 For a stratified space $X$ with a structure analogous to a cell
 complex, the first author
 \cite{1609.04500} defined an acyclic 
 topological category $C(X)$, called the \emph{face category} of $X$,
 whose classifying space is homotopy equivalent to $X$.
 On the other hand, there is a way to associate an $\infty$-category
 $\Exit(X)$, called the \emph{exit-path category}\footnote{A precise
 definition is given in \S\ref{exit-path}.} of $X$, to a
 stratified space satisfying certain conditions
 \cite{LurieHigherAlgebra}. Is $C(X)$ equivalent to $\Exit(X)$ as
 $\infty$-categories? 
\end{question}

\begin{question}
 \label{question:discrete_Morse}
 For a discrete Morse function $f$ or an acyclic partial matching
 on a regular CW complex $X$, Vidit 
 Nanda, Kohei Tanaka, and the first author \cite{1612.08429} constructed
 a poset-enriched category $C(f)$ whose classifying space\footnote{See
 \S\ref{discrete_Morse_theory} for the choice of classifying
 space of $2$-categories used here.} is homotopy equivalent to
 $X$. Does this classifying space have a ``cell decomposition'' analogous
 to smooth Morse theory? 
\end{question}

The original motivation for this work was Question
\ref{question:exit-path} posed by 
the second author during a series of talks by the first author at the
IBS Center for Geometry of Physics in Pohang. For any stratified space
$X$, $\Exit(X)$ can be defined as a simplicial set.
Before Question \ref{question:exit-path}, the first question
we need to address is if $\Exit(X)$ is a quasi-category.
Lurie proved as Theorem A.6.4 (1) in \cite{LurieHigherAlgebra} that
$\Exit(X)$ is a quasi-category if $X$ is conically
stratified\footnote{Definition \ref{definition:conically_stratified}.}.

\begin{question}
 \label{question:conically_stratified}
 When is a CW complex $X$ conically stratified?
\end{question}

It turns out that Question \ref{question:recover} is closely related to
this problem.
The stratified spaces in Question \ref{question:exit-path} are called
\emph{cylindrically normal stellar stratified spaces}\footnote{Precisely
speaking, CNSSS in this paper is slightly different from the one in
\cite{1609.04500}. See \S\ref{SSS_basics} for a precise definition.},
CNSSS for short, and an answer to Question \ref{question:recover} can be
given by using CNSSS.

\begin{theorem}
 \label{theorem:stratification}
 Let $C$ be an acyclic topological category with the space of objects
 $C_{0}$ having discrete topology. Suppose further that the space of
 morphisms $C(x,y)$ is compact Hausdorff for each pair $x,y\in C_{0}$
 and the set
 $P(C)_{<x}=\set{y\in C_{0}}{C(y,x)\neq\emptyset \text{ and } y\neq x}$
 is finite for each $x\in C_{0}$. 
 Then there exists a structure of CNSSS on the classifying space $BC$
 whose face category is isomorphic to $C$ as topological categories. 
\end{theorem}

Roughly speaking, CNSSS is a generalization of CW complex with
cells replaced by ``star-shaped cells''.
For a regular CW complex $X$, the stratification on the classifying
space of its face poset $F(X)$ obtained by Theorem \ref{theorem:stratification}
agrees with the original cell decomposition on $X$ under the standard
homeomorphism $X\cong BF(X)$. However, the use of
``star-shaped cells'' is essential for acyclic categories in general.
For example, consider the acyclic category $C$ depicted in Figure
\ref{figure:typical_acyclic_category}. 

 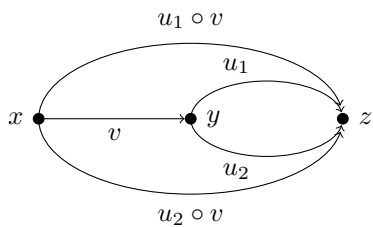
\begin{figure}[ht]
   \begin{center}
  \begin{tikzpicture}
   \draw [fill] (0,0) circle (2pt);
   \draw (-0.3,0) node {$x$};
   \draw [fill] (2,0) circle (2pt);
   \draw (2.3,0) node {$y$};
   \draw [fill] (4,0) circle (2pt);
   \draw (4.3,0) node {$z$};

   \draw [->] (0,0) -- (1.92,0);
   \draw (1,-0.2) node {$v$};

   \draw [->]  (2,0) arc [start angle=180, end angle=10,
   x radius=1, y radius=0.5];
   \draw (2.6,0.7) node {$u_1$};
   \draw [->]  (2,0) arc [start angle=180, end angle=350,
   x radius=1, y radius=0.5];
   \draw (2.6,-0.7) node {$u_2$};

   \draw [->]  (0,0) arc [start angle=180, end angle=10,
   x radius=2, y radius=1];
   \draw (2,1.3) node {$u_{1}\circ v$};
   \draw [->]  (0,0) arc [start angle=180, end angle=350,
   x radius=2, y radius=1];
   \draw (2,-1.3) node {$u_{2}\circ v$};
  \end{tikzpicture}
 \end{center}
  \caption{An acyclic category $C$}
  \label{figure:typical_acyclic_category}
 \end{figure}

 Its classifying space $BC$ and the stratification of $BC$ obtained
 by Theorem \ref{theorem:stratification} is shown in Figure
 \ref{figure:BC}. The middle stratum in the right-hand side of the
 equality is the $1$-cell $[x,y]$ with $x$ removed. Similarly in the
 right stratum, top and bottom edges of the ``hourglass'' are removed.
 The dotted arrows indicate inclusions of strata into 
 boundaries of closures of higher strata. By regarding the dotted arrows as
 morphisms and strata as objects, we recover the original category
 $C$. 

 \begin{figure}[ht]
    \begin{center}
  \begin{tikzpicture}
   \draw [fill,lightgray] (2,0) ellipse [x radius=2, y radius=1];
   \draw [fill,white] (3,0) ellipse [x radius=1, y radius=0.5];
   
   \draw [fill] (0,0) circle (2pt);
   \draw (-0.3,0) node {$x$};
   \draw [fill] (2,0) circle (2pt);
   \draw (2.3,0) node {$y$};
   \draw [fill] (4,0) circle (2pt);
   \draw (4.1,-0.4) node {$z$};

   \draw (0,0) -- (2,0);

   \draw (2,0) arc [start angle=180, end angle=0,
   x radius=1, y radius=0.5];
   \draw (2,0) arc [start angle=180, end angle=360,
   x radius=1, y radius=0.5];

   \draw (0,0) arc [start angle=180, end angle=4,
   x radius=2, y radius=1];
   \draw (0,0) arc [start angle=180, end angle=356,
   x radius=2, y radius=1];

   \draw (4.5,0) node {$=$};
   
   \draw [fill] (5,0) circle (2pt);
   \draw (5,-0.3) node {$x$};   

   \draw (5.5,0) node {$\cup$};
   
   \draw (6.02,0) -- (6.5,0);
   \draw (6.02,-0.3) node {$x$};   
   \draw [fill] (6,0) circle (2pt);
   \draw [fill,white] (6,0) circle (1pt);
   \draw [fill] (6.5,0) circle (2pt);
   \draw (6.5,-0.3) node {$y$};
   
   \draw (7,0) node {$\cup$};

   \draw [fill,lightgray] (7.5,1) -- (8,0) -- (8.5,1) --
   (7.5,1); 
   \draw (7.52,0.98) -- (8,0) -- (8.48,0.98);
   \draw [fill] (7.5,1) circle (2pt);
   \draw (7.5,1.3) node {$x$};   
   \draw [fill,white] (7.5,1) circle (1pt);
   \draw [fill] (8.5,1) circle (2pt);
   \draw (8.5,1.3) node {$y$};   
   \draw [fill,white] (8.5,1) circle (1pt);
   \draw [fill,lightgray] (7.5,-1) -- (8,0) -- (8.5,-1) --
   (7.5,-1); 
   \draw (7.52,-0.98) -- (8,0) -- (8.48,-0.98);
   \draw [fill] (7.5,-1) circle (2pt);
   \draw [fill,white] (7.5,-1) circle (1pt);
   \draw (7.5,-1.3) node {$x$};   
   \draw [fill] (8.5,-1) circle (2pt);
   \draw [fill,white] (8.5,-1) circle (1pt);
   \draw (8.5,-1.3) node {$y$};   
   \draw [fill] (8,0) circle (2pt);
   \draw (8.3,0) node {$z$};

   \draw [dotted,->] (5.2,0) -- (5.8,0);
   
   \draw [dotted,->] (6.25,0.1) -- (7.8,0.9);
   \draw [dotted,->] (6.25,-0.1) -- (7.8,-0.9);   
  \end{tikzpicture}
 \end{center}
  \caption{The classifying space $BC$ and its unstable stratification}
  \label{figure:BC}
 \end{figure}
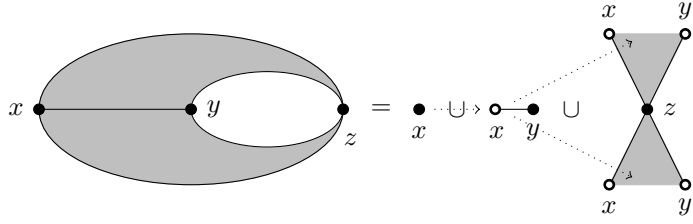

The stratification in Theorem \ref{theorem:stratification} is called the
\emph{unstable stratification} on 
$BC$. There is a dual stratification called the \emph{stable
stratification}. By combining these two stratifications, we obtain the
following result.

\begin{theorem}
 \label{theorem:conically_stratified}
 Let $C$ be an acyclic topological category satisfying the conditions of 
 Theorem \ref{theorem:stratification}. Suppose further that
 $P(C)_{>x}=\set{y\in C_{0}}{C(x,y)\neq\emptyset \text{ and }x\neq y}$
 is finite.
 Then the unstable stratification
 on $BC$ is conically stratified. Hence $\Exit(BC)$ is a quasi-category.
\end{theorem}

It is shown in \cite{1609.04500} that, when a CNSSS $X$ is a CW
complex, $BC(X)$ is homeomorphic to $X$, and
we obtain an answer to Question \ref{question:conically_stratified}. 

\begin{corollary}
 \label{corollary:exit-path_CNCW}
 If a finite CW complex $X$ has a structure of CNSSS, then $X$ is
 conically stratified, hence $\Exit(X)$ is a 
 quasi-category. 
\end{corollary}

Examples of CW complexes with such a structure are abundant.
Regular CW complexes are CNSSS. Among non-regular CW complexes, real and
complex projective spaces are typical examples. See Example 4.25, 4.26,
and 4.27 of \cite{1609.04500}. See \S4.2 of the paper for more examples.
PLCW complexes introduced by Alexander
Kirillov, Jr.~\cite{1009.4227} also provide non-regular examples of CNSSS.

Question \ref{question:discrete_Morse} is more directly related to
Question \ref{question:recover}. Since the classifying space in Question
\ref{question:discrete_Morse} is defined as the
classifying space of an acyclic topological category, Theorem
\ref{theorem:stratification} can be applied.

\begin{theorem}
 \label{theorem:discrete_Morse}
 For a discrete Morse function $f$ on a finite regular CW complex $X$,
 there exists a structure of a CNSSS on the classifying space $B^2C(f)$
 of the flow category $C(f)$ constructed in 
 \cite{1612.08429} satisfying the following conditions:
 \begin{enumerate}
  \item Strata are indexed by the set of critical cells of $f$.
  \item The face category is isomorphic to the topological category
	$BC(f)$ associated with the $2$-category $C(f)$.
 \end{enumerate}
\end{theorem}

The classifying space $B^2C(f)$ has a canonical structure of a cell
complex but this cell decomposition is much finer than the one obtained
from smooth Morse theory.
For example, consider the acyclic partial matching on the
boundary of a $3$-simplex $[v_{0},v_{1},v_{2}]$  in Figure
\ref{figure:height_function} which corresponds to a
``height function'' $h$.

\begin{figure}[ht]
\begin{center}
 \begin{tikzpicture}
  \draw [dotted] (0.5,2.5) -- (0,0);
  \draw (0,0) -- (1,1.5) -- (-1,2) -- (0,0);
  \draw (-1,2) -- (0.5,2.5) -- (1,1.5);
  \draw [fill] (0,0) circle (2pt);
  \draw [fill] (0.5,2.5) circle (2pt);
  \draw [fill] (-1,2) circle (2pt);
  \draw [fill] (1,1.5) circle (2pt);

	  
  \draw (0,-0.4) node {$v_{0}$};
  \draw (1.4,1.4) node {$v_{1}$};
  \draw (0.6,2.9) node {$v_{2}$};
  \draw (-1.4,2.1) node {$v_{3}$};
	  
	  
  \draw [->] (1,1.3) -- (0.8,1);
  \draw [->] (0.48,2.3) -- (0.38,1.8);
  \draw [->] (-1,1.8) -- (-0.8,1.4);
  \draw [->] (0,1.75) -- (0,1.4);
  \draw [->] (0.75,2) -- (0.65,1.6);
  \draw [->] (-0.25,2.25) -- (-0.15,1.75);
 \end{tikzpicture}
\end{center}
 \caption{An acyclic partial matching on $\partial[v_{0},v_{1},v_{2}]$}
 \label{figure:height_function}
\end{figure}
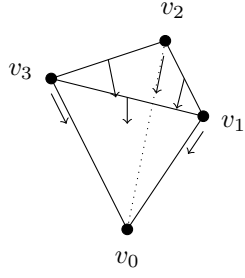

Matched pairs are indicated by arrows. For example, the $0$-simplex
$[v_{1}]$ is matched with a $1$-simplex $[v_{0},v_{1}]$ and the
$1$-simplex $[v_{1},v_{2}]$ is matched with a $2$-simplex
$[v_{0},v_{1},v_{2}]$. 
The $2$-category $C(h)$ has
two objects corresponding to critical simplices, i.e.~the top face
$[v_{1},v_{2},v_{3}]$ and 
the bottom vertex $[v_{0}]$.
As is shown in Example 3.17 of \cite{1612.08429}, the category (poset) of
morphisms from $[v_{0}]$ to 
$[v_{1},v_{2},v_{3}]$ is isomorphic to the face poset of
$\partial[v_{1},v_{2},v_{3}]$, and hence its classifying space is the
boundary of a hexagon. Thus the classifying space of $C(h)$ is a regular
cell complex 
consisting of two $0$-cells, six $1$-cells, and six $2$-cells. However
the cell decomposition of $S^2$ we usually obtain from a height function
is the minimal cell decomposition $S^2=e^{0}\cup e^2$. We should glue
six $2$-cells, six $1$-cells, and one of the $0$-cells together
to obtain a single $2$-cell $e^2$ so that we have
$B^2C(h)=e^{0}\cup e^2$.
The motivation for Question \ref{question:discrete_Morse} is to generalize
this construction and Theorem \ref{theorem:discrete_Morse} solves the
problem. 

From the viewpoint of topological combinatorics, Theorem
\ref{theorem:stratification} is closely related to the well-known
characterization of the face poset of a regular CW complex by
Bj{\"o}rner. 

\begin{theorem}[\cite{Bjorner1984}]
 The category of finite regular CW complexes is equivalent to
 the category of finite CW posets via the face poset functor.
\[
 F: \category{RegCW}^{f} \rarrow{\simeq} \category{CWPoset}^{f}.
\]
\end{theorem}

Let $\category{AcycTopCat}$ and $\category{AcycTopCat}^{lf}_{\cpt,\Haus}$
denote the category of acyclic topological categories and the
full subcategory of those satisfying the conditions of Theorem
\ref{theorem:stratification}, respectively. 
The face poset functor for regular CW complexes has been extended to the
face category functor in \cite{1609.04500}
\begin{equation}
 C: \category{CNSSS} \rarrow{} \category{AcycTopCat}
  \label{equation:face_category}
\end{equation}
from the category of CNSSSs to the category of acyclic topological
categories. 
Our construction in Theorem \ref{theorem:stratification} is a right
inverse to this face category functor when restricted to appropriate
subcategories. 

\begin{theorem}
 \label{theorem:equivalence}
 Let $\category{CNCW}$ be the full subcategory of $\CNSSS$ consisting of
 cylindrically normal CW stellar complexes. Then the
 restriction of the face category functor (\ref{equation:face_category})
 \[
 C: \category{CNCW} \rarrow{} \category{AcycTopCat}^{lf}_{\cpt,\Haus}  
 \]
 is an equivalence of categories.
\end{theorem}

This paper is organized as follows.
\begin{itemize}
 \item \S\ref{recollections} is preliminary. We fix notion and
       terminology for nerves and classifying spaces
       of small categories, including simplicial techniques.
 \item \S\ref{SSS} collects necessary materials for stellar stratified
       spaces from \cite{1609.04500} with some generalizations and
       extensions. 
 \item \S\ref{stellar_stratification_on_BC} is the main part. Theorems
       mentioned above are proved.       
 \item We conclude this paper by a couple of remarks and comments in
       \S\ref{remarks}. 
\end{itemize}

\subsection{Acknowledgments}

This project started when the authors were invited to the IBS Center
for Geometry and Physics in Pohang in December, 2016.
We would like to thank the center for invitation and the nice
working environment.

The contents of this paper was presented by the first author during the 7th
East Asian Conference on Algebraic Topology held at Mohali, India, in
December, 2017.
He is grateful to the local organizers for the invitation to the
conference and the hospitality of ISSER, Mohali.

The authors would like to thank the anonymous referee whose valuable
suggestions improved expositions and made this paper more readable.

\section{Recollections}
\label{recollections}


\subsection{Simplicial Terminology}
\label{simplicial}

Here we fix notation and terminology for simplicial homotopy theory.

\begin{definition} \hspace*{\fill}
 \begin{enumerate}
  \item The category of isomorphism classes of finite totally ordered
	sets and order preserving maps is denoted by $\Delta$. 
  \item The wide subcategory of injective maps is denoted by
	$\Delta_{\inj}$.
  \item Objects in these categories are represented by
	$[n]=\{0<1<\cdots<n\}$ for nonnegative integers $n$.
  \item A \emph{simplicial space} is a functor $X:\Delta^{\op}\to\Top$,
	where $\Top$ is the category of topological spaces and continuous
	maps. 
  \item Dually a \emph{cosimplicial space} is a functor $Y:\Delta\to\Top$.
  \item A functor $X:\Delta_{\inj}^{\op}\to \Top$ is called a
	\emph{$\Delta$-space}. 
  \item For a nonnegative integer $n$, the \emph{geometric $n$-simplex}
	is defined by
	\[
	\DDelta^{n} =
	\set{(t_{0},\ldots,t_{n})\in\R^{n+1}}{\sum_{i=0}^{n} t_{i}=1,
	t_{i}\ge 0}.
	\]
  \item The geometric realization of a simplicial space $X$ is denoted
	by $|X|$, while the geometric realization of a
	$\Delta$-space\footnote{$\Delta$-spaces are sometimes called
	semisimplicial spaces, e.g.~in \cite{LurieHigherToposTheory} and
	\cite{1705.03774}.} is
	denoted by $\|X\|$.
 \end{enumerate}

\end{definition}
\subsection{Nerves and Classifying Spaces}
\label{nerve}

Let us first recall basic properties of the classifying space
construction.
For a topological category $C$, the spaces of objects and morphisms are
denoted by $C_{0}$ and $C_{1}$, respectively. The space of morphisms
from $x$ to $y$ is denoted by $C(x,y)$.
The space $N_{k}(C)$ of $k$-chains in $C$ is defined to be the set of
all functors $[k]\to C$ topologized as a subspace of $C_1^{k}$ under the
identification 
\[
 N_{k}(C) \cong \set{(u_{k},u_{k-1},\ldots,u_{1})\in
 C_{1}^{k}}{x_{0} \rarrow{u_{1}} x_{1} \rightarrow
 \cdots \rightarrow x_{k-1} \rarrow{u_{k}} x_{k}}.
\]
With this notation the structure of a category is given by a pair of
maps 
\begin{align*}
 \circ & : N_{2}(C) \rarrow{} C_{1} \\
 \iota & : C_{0} \rarrow{} C_{1}
\end{align*}
satisfying the associativity and the unit conditions.

The collection $N(C)=\{N_{k}(C)\}_{k\ge 0}$ can be made into a simplicial
space, called the \emph{nerve} of $C$.
The geometric realization of the nerve is denoted by $BC$ and is called
the \emph{classifying space} of $C$. The defining quotient map is
denoted by
\begin{equation}
 p_{C}:\coprod_{k\ge 0} N_{k}(C)\times\DDelta^{k} \rarrow{} BC.
  \label{quotient_for_BC}
\end{equation}

We are mainly interested in acyclic categories.

\begin{definition}
\label{definition:acyclic_category}
 A topological category $C$ is called \emph{acyclic} if the following
 conditions are satisfied:
 \begin{enumerate}
  \item For any pair of distinct objects $x,y\in C_{0}$, either $C(x,y)$
	or $C(y,x)$ is empty.
  \item For any object $x\in C_{0}$, $C(x,x)$ consists only of the
	identity morphism.
  \item Regard $C_{0}$ as the subspace of identity morphisms in
	$C_{1}$. Then
	$C_{1}= C_{0}\amalg (C_{1}\setminus C_{0})$ as topological spaces.
 \end{enumerate}
 For $x,y\in C_{0}$, define $x\le y$ if and only if
 $C(x,y)\neq\emptyset$. When $C$ is acyclic, $C_{0}$ becomes a poset
 under this relation. This poset is denoted by $P(C)$.

 An acyclic topological category $C$ is called a
 \emph{topological poset} if it is a poset when the topology is
 forgotten. 
\end{definition}


The last condition in the definition of acyclicity simplifies the
description of the classifying space $BC$.

\begin{lemma}
 \label{lemma:nerve_of_acyclic_category}
 For an acyclic topological category $C$, define
 \[
  \nN_{k}(C) = N_{k}(C) \setminus \bigcup_{i} s_{i}N_{k-1}(C),
 \]
 where $s_i: N_{k-1}(C)\to N_{k}(C)$ is the $i$-th degeneracy operator.
 Then the collection $\nN(C)=\{\nN_{k}(C)\}$ together with restrictions
 of the face operators in $N(C)$ forms a
 $\Delta$-space and
 the canonical inclusion $\nN(C)\hookrightarrow N(C)$ induces a
 homeomorphism $\left\|\nN(C)\right\| \cong |N(C)|=BC$.
\end{lemma}

The $\Delta$-space $\nN(C)$ is called the \emph{nondegenerate nerve} of
$C$.

\section{Stellar Stratified Spaces and Their Face Categories}
\label{SSS}

The notion of stellar stratified spaces was introduced by the first
author in \S2.4 of \cite{1609.04500}.
Roughly speaking, a stellar stratification on a topological space $X$ is
a stratification of $X$ together with identifications of strata
with ``star-shaped cells''.
In \cite{1609.04500}, these star-shaped cells are defined as subspaces
of closed disks. Here we extend the definition by using cones on
stratified spaces.

\subsection{Stratifications by Posets}
\label{stratification}

Before we give a definition of stellar stratified spaces, we need to
clarify what we mean by a stratification, since the meaning of
stratification varies in the literature. 
Generally, a stratification of a topological space $X$ is a
decomposition of $X$ into a mutually disjoint union of subspaces,
satisfying some boundary conditions. The boundary conditions are often
described in terms of posets. 

Decomposing a space $X$ into a union of mutually disjoint subspaces
\[
 X = \bigcup_{\lambda\in\Lambda} e_{\lambda}.
\]
is equivalent to giving a surjective map $\pi : X\to \Lambda$ with
$\pi^{-1}(\lambda)=e_{\lambda}$. 
When can we call such a decomposition a stratification? Several
conditions have been proposed. Recall that a partial order on $\Lambda$
generates a topology, called the Alexandroff topology, on $\Lambda$.
Here we use the following definition from \cite{1609.04500}.

\begin{definition}
 \label{definition:stratification}
 A \emph{stratification} of a topological space $X$ by a poset $\Lambda$
 is an open continuous map $\pi : X\to\Lambda$ such that
 $\pi^{-1}(\lambda)$ is connected and locally closed for each
 $\lambda\in\Ima\pi$, where $\Lambda$ is equipped with the Alexandroff
 topology.
 Such a pair $(X,\pi)$ is called a \emph{$\Lambda$-stratified space}
 when $\pi$ is surjective.
 
 The image of $\pi$ is denoted by $P(X)$ and is called the \emph{face
 poset} of $X$. It is regarded as a full subposet of $\Lambda$.
 For $\lambda\in P(X)$,
 $\pi^{-1}(\lambda)$ is called the \emph{stratum} indexed by $\lambda$
 and is denoted by $e_{\lambda}$.
\end{definition}

Nowadays many authors use simpler definitions. The simplest one
only assumes the continuity of $\pi$. 
On the other hand, we need to impose some additional ``niceness''
conditions to have a good hold on the topological properties of
stratified spaces.

In order to understand the meanings of these conditions, let us compare
the following five conditions.

\begin{enumerate}
 \item \label{Andrade} $\pi$ is continuous.
 \item \label{open} $\pi$ is an open map.
 \item \label{Tamaki} For any $\mu,\lambda\in\Lambda$,
       $e_{\mu}\subset \overline{e_{\lambda}}$ if and 
       only if $\mu\le\lambda$. 
 \item \label{normal} For any $\mu,\lambda\in\Lambda$,
       $e_{\mu}\cap\overline{e_{\lambda}}\neq\emptyset$
       implies $e_{\mu}\subset \overline{e_{\lambda}}$, or equivalently,
       for any $\lambda\in\Lambda$, 
       $\displaystyle\overline{e_{\lambda}}=\bigcup_{e_{\mu}\cap\overline{e_{\lambda}}\neq\emptyset} e_{\mu}$. 
 \item \label{Hiro} For any closed subset $C\subset\Lambda$, 
       $\displaystyle\bigcup_{\lambda\in C}\overline{e_{\lambda}}$ is
       closed. 
\end{enumerate}

The condition \ref{Andrade} seems to be the most popular one these
days. It is used in Andrade's thesis 
\cite{AndradeThesis}, Lurie's book \cite{LurieHigherAlgebra}, papers 
concerning factorization homology \cite{1409.0501,1409.0848}, and so on.
The combination \ref{Andrade}+\ref{open} is used by the first author
\cite{1609.04500}. 
The condition \ref{Tamaki} is used by the first author in a series of
talks at the Center for Geometry and Physics in Pohang.
Cell complexes satisfying the condition \ref{normal} are usually
called \emph{normal} \cite{Lundell-Weingram}.
This condition has been
also known as ``the axiom of the frontier'' in the classical
stratification theory due to Thom \cite{Thom69} and Mather
\cite{MatherNotes}. 
The condition \ref{Hiro} was mentioned by the second author during the
above mentioned talks in Pohang.

The following fact is stated and proved in \cite{1609.04500} as Lemma 2.3.

\begin{proposition}
 \label{Andrade+open=Tamaki}
 Suppose \ref{Andrade} is satisfied. Then \ref{open} is equivalent to
 \ref{Tamaki}. 
\end{proposition}

For the convenience of the reader, we record a proof.
We use the following fact.

\begin{lemma}
 \label{open_and_continuous}
 A map $f:X\to Y$ between topological spaces is open if and only if 
 $f^{-1}(\overline{B})\subset\overline{f^{-1}(B)}$ for any
 $B\subset Y$. In particular, $f$ is open and continuous if and only if 
 $f^{-1}(\overline{B})=\overline{f^{-1}(B)}$ for any
 $B\subset Y$.
\end{lemma}

\begin{proof}
 Suppose $f$ is open. Suppose further that
 $f^{-1}(\overline{B})\not\subset \overline{f^{-1}(B)}$. Then there 
 exists $x\in f^{-1}(\overline{B})\setminus \overline{f^{-1}(B)}$.
 In other words, $f(x)\in \overline{B}$ and there exists an open
 neighborhood $U$ of $x$ such that $U\cap f^{-1}(B)=\emptyset$. The
 first condition implies that $V\cap B\neq\emptyset$ for any open
 neighborhood $V$ of $f(x)$ in $Y$. The second condition implies that
 $f(U)\cap B=\emptyset$. Since $f$ is open, $f(U)$ is an open
 neighborhood of $f(x)$ and this is a contradiction.

 Conversely suppose that
 $f^{-1}(\overline{B})\subset\overline{f^{-1}(B)}$
 for any $B\subset Y$.
 For an open set $U\subset X$, we have
 \[
  f^{-1}(\overline{Y\setminus f(U)}) \subset \overline{f^{-1}(Y\setminus
 f(U))} \subset \overline{X\setminus U}=X\setminus U,
 \]
 which implies that $f^{-1}(\overline{Y\setminus f(U)})\cap U=\emptyset$
 or $\overline{Y\setminus f(U)}\cap f(U)=\emptyset$. Thus $f(U)$ is an
 open set.
\end{proof}

\begin{proof}[Proof of Proposition \ref{Andrade+open=Tamaki}]
 Suppose $\pi$ is continuous. When $\pi$ is open, Lemma
 \ref{open_and_continuous} implies that
 $\pi^{-1}(\overline{\{\lambda}\})=\overline{\pi^{-1}(\lambda)}$.
 Thus $e_{\mu}\subset\overline{e_{\lambda}}$ if and only if
 $\pi^{-1}(\mu)\subset \pi^{-1}(\overline{\{\lambda\}})$, which is
 equivalent to $\mu\in\overline{\{\lambda\}}$. By the definition of the
 Alexandroff topology, this is equivalent to $\mu\le\lambda$.
\end{proof}

Proposition \ref{Andrade+open=Tamaki} suggests a close relationship
between the normality condition and the openness of $\pi$.
In fact, we have the following. 

\begin{proposition}
 \label{Andrade+Tamaki->normal}
 The conditions \ref{Andrade} and \ref{Tamaki} imply the condition
 \ref{normal}.
\end{proposition}

\begin{proof}
 When $\pi$ is continuous, we have
 \[
  \overline{e_{\lambda}}=\overline{\pi^{-1}(\{\lambda\})} \subset
 \pi^{-1}(\overline{\{\lambda\}}).
 \]
 If $e_{\mu}\cap\overline{e_{\lambda}}\neq\emptyset$, then
 $e_{\mu}\cap\pi^{-1}(\overline{\{\lambda\}})\neq\emptyset$, or
 $\mu\in\overline{\{\lambda\}}$, or $\mu\le\lambda$. 
 By \ref{Tamaki}, this is equivalent to
 $e_{\mu}\subset\overline{e_{\lambda}}$. 
\end{proof}

\begin{corollary}
 \label{Tamaki_stratification_implies_normal}
 Any stratification in the sense of Definition
 \ref{definition:stratification} is normal in the sense of Definition
 2.6 of \cite{1609.04500}.
\end{corollary}

The condition \ref{Andrade} obviously implies \ref{Hiro}.
For the converse, we have the following.

\begin{proposition}
 \label{Tamaki+normal+Hiro->Andrade}
 The conditions \ref{Tamaki}, \ref{normal}, and \ref{Hiro} imply
 \ref{Andrade}.  
\end{proposition}

\begin{proof}
 Suppose $C\subset\Lambda$ is closed. Then we have
 \[
  \overline{\pi^{-1}(C)} \supset \bigcup_{\lambda\in C}
 \overline{\pi^{-1}(\lambda)} \supset \bigcup_{\lambda\in C}
 \pi^{-1}(\lambda) = \pi^{-1}(C).
 \]
 By the condition \ref{Hiro}, we obtain
 \[
 \displaystyle\overline{\pi^{-1}(C)}=\bigcup_{\lambda\in
 C}\overline{\pi^{-1}(\lambda)} = \bigcup_{\lambda\in C}
 \overline{e_{\lambda}}. 
 \]
 On the other hand, the conditions \ref{Tamaki} and \ref{normal} allow us to
 write each $\overline{e_{\lambda}}$ as
 $\bigcup_{\mu\le\lambda} e_{\mu}$. Since $C$ is closed, $\lambda\in C$
 and $\mu\le\lambda$ imply that $\mu\in C$ and we have
 \[
 \bigcup_{\lambda\in C}\overline{e_{\lambda}} = \bigcup_{\lambda\in
 C}\bigcup_{\mu\le\lambda} e_{\mu} = \bigcup_{\lambda\in C}
 e_{\lambda} = \pi^{-1}(C)
 \]
 and $\pi^{-1}(C)$ is shown to be closed.
\end{proof}

For a cell complex $X$, define $P(X)$ to be the set of cells of
$X$. Define a partial order $\le$ on $P(X)$ by saying that
$e\le e'$ if and only if $e\subset \overline{e'}$.
We have a map $\pi:X\to P(X)$ which assigns the
unique cell $\pi(x)$ containing $x$ to $x\in X$.
 
In general, this is not a stratification in the sense of Definition
\ref{definition:stratification}. 
Proposition \ref{Tamaki+normal+Hiro->Andrade} implies, however, that 
$\pi$ is a stratification if $X$ is normal.

\begin{corollary}
 \label{corollary:cell_decomposition}
 For a normal CW complex $X$, the map $X\to P(X)$ defined by the cell
 decomposition is open and
 continuous, hence is a stratification in the sense of Definition
 \ref{definition:stratification}.  
\end{corollary}

In particular, the geometric realization of a simplicial complex is a
stratified space.

\begin{example}
 \label{example:simplicial_complex}
 Let $K$ be an ordered simplicial complex. Then the geometric
 realization $\|K\|$ has a structure of regular CW complex whose cells
 are indexed by the simplices of $K$. Since regular CW complexes are
 normal, by Corollary \ref{corollary:cell_decomposition}, the cell
 decomposition is a stratification. 
 This stratification can be generalized to $\Delta$-spaces, i.e.\
 simplicial spaces without degeneracies.
 For a $\Delta$-space $X$, the \emph{simplicial stratification}
 \[
 \pi_{X} : \|X\| =
 \quotient{\coprod_{n\ge 0}X_{n}\times\DDelta^{n}}{\sim} 
 \rarrow{} \coprod_{n\ge 0}  X_{n}
 \]
 is defined by $\pi_{X}([x,\bm{t}])=x$ when $\bm{t}\in\inte\DDelta^n$,
 where the topology of each $X_{n}$ is forgotten and the partial order
 on $\coprod_{n\ge 0}X_{n}$ is defined by
 \[
 x\le y \Longleftrightarrow \exists u\in \Delta_{\inj}([m],[n]) \text{
 such that } X(u)(y)=x
 \]
 for $x\in X_{n}$, $y\in X_{m}$.
 See Example 3.16 of \cite{1609.04500}, for details.
\end{example}

The definition of morphisms between stratified spaces should be obvious.

\begin{definition}
 A \emph{morphism} of stratified spaces from $\pi_{X}:X\to P(X)$ to
 $\pi_{Y}:Y\to P(Y)$ is a pair of a continuous map $f:X\to Y$ and a
 morphism of posets $P(f): P(X)\to P(Y)$ making the obvious diagram
 commutative. 

 A morphism $f:X\to Y$  of stratified spaces is called \emph{strict} if 
 $f(e_{\lambda})=e_{P(f)(\lambda)}$. 
\end{definition}

As is the case of cell complexes, the CW condition plays an essential
role when we study topological and homotopy-theoretic properties.

\begin{definition}
 \label{definition:CW}
 A stratification $\pi$ on a topological space $X$ is said to be
 \emph{CW} if it satisfies the following conditions:
 \begin{enumerate}
  \item (Closure Finite) For each stratum $e_{\lambda}$, the boundary
	$\partial e_{\lambda}$ is covered by a finite number of strata.
  \item (Weak Topology) $X$ has the weak topology determined by the
	covering $\{\overline{e_{\lambda}}\}_{\lambda\in P(X)}$.
 \end{enumerate}
\end{definition}

\subsection{Joins and Cones}
\label{join}

We need to make use of the join of stratified spaces in order to define
stellar structures.

Recall that, for topological spaces $X$ and $Y$, the \emph{join}
$X\star Y$ is defined to be the quotient space
 \[
  X\star Y = X\times [0,1]\times Y/_{\sim}
 \]
 where the equivalence relation $\sim$ is generated by the following two
 types of relations:
 \begin{enumerate}
  \item $(x,0,y)\sim (x,0,y')$ for all $x\in X$ and $y,y'\in Y$.
  \item $(x,1,y)\sim (x',1,y)$ for all $x,x'\in X$ and $y\in Y$.
 \end{enumerate}
 The class represented by $(x,t,y) \in X\times[0,1]\times Y$ is denoted
 by $(1-t)x+ty$. 

\begin{definition}
 \label{definition:stratification_on_join}
 When $X$ and $Y$ are stratified by maps $\pi_{X}:X\to P(X)$ and
 $\pi_{Y}:Y\to P(Y)$, define
 \[
  \pi_{X}\star \pi_{Y} : X\star Y \rarrow{} P(X) \amalg P(X)\times P(Y)
 \amalg P(Y) 
 \]
 by
 \[
 (\pi_{X}\star \pi_{Y})((1-t)x+ty) =
 \begin{cases}
  \pi_{X}(x), & t=0 \\
  (\pi_{X}(x),\pi_{Y}(y)), & 0<t<1 \\
  \pi_{Y}(y), & t=1.
 \end{cases}
 \]
\end{definition}

\begin{lemma}
 For stratified spaces $\pi_{X}:X\to P(X)$ and
 $\pi_{Y}:Y\to P(Y)$, define a partial order on
 $P(X) \amalg P(X)\times P(Y) \amalg P(Y)$ by the following rule.
 \begin{enumerate}
  \item $P(X)$ and $P(Y)$ are full subposets.
  \item $P(X)\times P(Y)$ with the product partial order is a full
	subposet. 
  \item $\lambda < (\lambda,\mu)$ for all $\lambda\in P(X)$ and
	$\mu\in	P(Y)$. 
  \item $\mu < (\lambda,\mu)$ for all $\lambda\in P(X)$ and
	$\mu\in	P(Y)$. 
 \end{enumerate}
 Let us denote this poset as $P(X)\star P(Y)$.
 Then the map $\pi_{X}\star \pi_{Y}:X\star Y\to P(X)\star P(Y)$ defines
 a stratification on the join $X\star Y$.
\end{lemma}

\begin{proof}
 Let us denote the strata in $X$ and $Y$ by $e^{X}_{\lambda}$ and
 $e^{Y}_{\mu}$ for $\lambda\in P(X)$ and $\mu\in P(Y)$, respectively. 
 Regard $X$ and $Y$ as subspaces of $X\star Y$. Then we have a
 decomposition
 \[
  X\star Y = \coprod_{\lambda\in P(X)} e^{X}_{\lambda} \amalg
 \coprod_{(\lambda,\mu) \in P(X)\times P(Y)} e^{X\star
 Y}_{(\lambda,\mu)} \amalg \coprod_{\mu\in P(Y)} e^{Y}_{\mu}, 
 \]
 where
 \[
  e^{X\star Y}_{(\lambda,\mu)} = (\pi_{X}\star
 \pi_{Y})^{-1}(\lambda,\mu) = e^{X}_{\lambda}\times (0,1)\times
 e^{Y}_{\mu}. 
 \]
 It remains to show that $\pi_{X}\star\pi_{Y}$ is open and continuous.
 By Proposition \ref{Tamaki+normal+Hiro->Andrade} and Proposition
 \ref{Andrade+open=Tamaki}, it suffices to show that
 $\pi_{X}\star\pi_{Y}$ satisfies the conditions \ref{Tamaki},
 \ref{normal}, and \ref{Hiro} in \S\ref{stratification}.

 By Proposition \ref{Andrade+open=Tamaki} and Proposition
 \ref{Andrade+Tamaki->normal}, $\pi_{X}$  and $\pi_{Y}$ satisfy
 \ref{Tamaki}, \ref{normal}, and \ref{Hiro}.
 By the definition of the partial order on $P(X)\star P(Y)$,
 $\pi_{X}\star\pi_{Y}$ satisfies \ref{Tamaki}.
 The condition \ref{normal} follows from
 the fact that
 $\overline{e^{X}_{\lambda}\star e^{Y}_{\mu}}=\overline{e^{X}_{\lambda}}\star\overline{e^{Y}_{\mu}}$. 
 Let $C$ be a closed subset of $P(X)\star P(Y)$. It decomposes as
 $C=C_{X}\cup C_{X,Y}\cup C_{Y}$ with $C_{X}\subset P(X)$,
 $C(Y)\subset P(Y)$, and $C_{X,Y}\subset P(X)\times P(Y)$. We have
\begin{align*}
 \bigcup_{\nu\in C} \overline{(\pi_{X}\star \pi_{Y})^{-1}(\nu)} & =
 \bigcup_{\lambda\in C_{X}} \overline{e^{X}_{\lambda}} \cup
 \bigcup_{(\lambda,\mu)\in C_{X,Y}}
 \overline{e^{X}_{\lambda}\times (0,1)\times e^{Y}_{\mu}} \cup
 \bigcup_{\mu\in C_{Y}} \overline{e^{Y}_{\mu}} \\
 & = \bigcup_{\lambda\in C_{X}} \overline{e^{X}_{\lambda}} \cup
 \bigcup_{(\lambda,\mu)\in C_{X,Y}}
 \overline{e^{X}_{\lambda}}\star \overline{e^{Y}_{\mu}} \cup
 \bigcup_{\mu\in C_{Y}} \overline{e^{Y}_{\mu}}.
\end{align*}
 Let $p_{X\star Y}: X\times[0,1]\times Y\to X\star Y$ be the defining
 quotient map. Then we have
 \begin{align*}
 & p_{X\star Y}^{-1}\left(\bigcup_{\lambda\in C_{X}}
  \overline{e^{X}_{\lambda}} \cup \bigcup_{(\lambda,\mu)\in C_{X,Y}}
  \overline{e^{X}_{\lambda}}\star \overline{e^{Y}_{\mu}}
  \cup
  \bigcup_{\mu\in C_{Y}} \overline{e^{Y}_{\mu}}\right) \\
  & = \bigcup_{\lambda\in C_{X}}
  p_{X\star Y}^{-1}\left(\overline{e^{X}_{\lambda}}\right) \cup
  \bigcup_{(\lambda,\mu)\in C_{X,Y}}
  p_{X\star Y}^{-1}\left(\overline{e^{X}_{\lambda}}\star
  \overline{e^{Y}_{\mu}}\right) \cup
  \bigcup_{\mu\in C_{Y}}
  p_{X\star Y}^{-1}\left(\overline{e^{Y}_{\mu}}\right) \\ 
  & = \bigcup_{\lambda\in C_{X}}
  \overline{e^{X}_{\lambda}}\times\{0\}\times Y 
  \cup \bigcup_{(\lambda,\mu)\in C_{X,Y}}  
  \overline{e^{X}_{\lambda}}\times[0,1]\times\overline{e^{Y}_{\mu}}
  \cup
  \bigcup_{\mu\in C_{Y}}
  X\times\{1\}\times\overline{e^{Y}_{\mu}}.
 \end{align*}
 Since $C_{X}$, $C_{X,Y}$, and $C_{Y}$ are closed in $P(X)$,
 $P(X)\times P(Y)$, and $P(Y)$, respectively,
 this is closed in $X\times [0,1]\times Y$.
 Hence $\bigcup_{\nu\in C} \overline{(\pi_{X}\star \pi_{Y})^{-1}(\nu)}$
 is closed in $X\star Y$.
\end{proof}

In particular, we have the (closed) cone construction on stratified
spaces. 

\begin{definition}
 \label{defn:cone}
 For a stratified space $\pi_{X}:X\to\Lambda$, the \emph{cone on $X$}
 is defined by
 \[
  \cone(X)=\{\ast\}\star X.
 \]
 The complement
 \[
  \cone(X)\setminus X\times\{1\} = \set{(1-t)\ast + tx\in \cone(X)}{0\le t<1}
 \]
 is called the \emph{open cone} and is denoted by $\cone^{\circ}(X)$. 
\end{definition}

\begin{remark}
 Make a set $\{b,i\}$ into a poset by $b<i$. Then the face poset of
 $\cone(X)$ can be identified with $\Lambda\times\{b,i\}\amalg\{\ast\}$.
 The element $\ast$ is the unique minimal element in
 $\Lambda\times\{i\}\amalg\{\ast\}$ and is unrelated to elements in
 $\Lambda\times\{b\}$.
 
 With this identification, the stratification on $\cone(X)$ 
 \[
  \pi_{\cone(X)} : \cone(X) \rarrow{} \Lambda\times\{b,i\}\amalg\{\ast\}
 \]
 is given by
 \[
 \pi_{\cone(X)}((1-t)\ast + tx) =
 \begin{cases}
  \ast, & t=0 \\
  (\pi_{X}(x),i), & 0<t<1 \\
  (\pi_{X}(x),b), & t=1.
 \end{cases}
 \]

 The face poset of the open cone $\cone^{\circ}(X)$, which is
 $\Lambda\times\{i\}\amalg \{\ast\}$, is denoted by
 $\Lambda^{\lhd}$ in Lurie's book \cite{LurieHigherAlgebra}.
\end{remark}

\subsection{Stellar Stratified Spaces}
\label{SSS_basics}

Recall that a characteristic map of a cell $e$ in a cell complex $X$ is
a surjective continuous map $\varphi: D^{\dim e}\to \overline{e}$ which
is a homeomorphism onto $e$ when restricted to the interior of $D^{\dim e}$.
The notion of stellar structure was introduced in \cite{1609.04500} by
replacing disks by stellar cells, in which a stellar cell was defined as a
subspace of a closed disk $D^{N}$ obtained by taking the join of the
center and a subspace $S$ of the boundary. In other words, it is a cone
$\cone(S)$ on $S$ embedded in $D^{N}$. 
Here we do not require this embeddability condition.

\begin{definition}
 Let $S$ be a stratified space. A subset $D\subset \cone(S)$ is called
 an \emph{aster} if, for any $x\in D$ with $x=(1-t)*+ty$,
 $(1-t')*+t'y\in D$ for any $0\le t'\le t$.
 The subset $S\cap D$ is called the boundary of $D$ and is denoted by
 $\partial D$. The complement $D\setminus\partial D$ is called the
 \emph{interior} and is denoted by $\inte(D)$.
 An aster $D$ is called \emph{thin} if $D=\{*\}\star\partial D$. 
\end{definition}

\begin{definition}
 Let $\pi : X\to \Lambda$ be a stratified space. A \emph{stellar
 structure} on a stratum $e_{\lambda}$ is a morphism of stratified
 spaces
 $\varphi_{\lambda} : D_{\lambda} \rarrow{} \overline{e_{\lambda}}$,
 for an aster $D_{\lambda}\subset \cone(S_{\lambda})$, which is a
 quotient map and whose restriction 
 $\varphi_{\lambda}|_{\inte(D_{\lambda)}}: \inte(D_{\lambda})\to e_{\lambda}$  
 is a homeomorphism.
 If $S_{\lambda}$ is a stratified subspace of a stratification of a
 sphere $S^{n-1}$ and $\inte(D_{\lambda})=\inte(D^{n})$, then the
 stellar structure is called a \emph{cell structure}.

 A \emph{stellar stratified space} is a triple $(X,\pi_{X},\Phi_{X})$ of
 a topological space $X$, a stratification $\pi_{X}:X\to P(X)$, and a
 collection of stellar structures
 $\Phi_{X}=\{\varphi_{\lambda}\}_{\lambda\in P(X)}$ on strata such that,
 for each stratum $e_{\lambda}$,
 $\partial e_{\lambda}=\overline{e_{\lambda}}\setminus e_{\lambda}$ is
 covered by strata indexed by $P(X)_{<\lambda}=\set{\mu\in P(X)}{\mu<\lambda}$.
 When all stellar structures are cell structures, it is called a
 \emph{cellular stratified space}.
 
 A stratum $e_{\lambda}$ in a stellar stratified space is called
 \emph{thin} if the domain $D_{\lambda}$ of the stellar structure
 $\varphi_{\lambda}: D_{\lambda}\to \overline{e_{\lambda}}$ is a thin
 aster.
 A stellar stratified space is called a \emph{stellar complex} if all
 strata are thin. 
\end{definition}

\begin{remark}
 We do not require the restriction $\varphi|_{\inte(D_{\lambda})}$
 to be an isomorphism of stratified spaces.
\end{remark}

\begin{example}
 Consider a cell $e_{\lambda}$ in a normal CW complex $X$. The
 characteristic map
 $\varphi_{\lambda}: D^{\dim e_{\lambda}}\to\overline{e_{\lambda}}$
 defines a stellar structure on $e_{\lambda}$ by setting
 $S_{\lambda}=\partial D^{\dim e_{\lambda}}=S^{\dim e_{\lambda}-1}$.
 The stratification on $S_{\lambda}$ is defined by taking connected
 components of strata obtained by pulling back the
 stratification on $\partial e_{\lambda}$.
 Thus any normal CW complex can be regarded as a stellar stratified
 space. 
\end{example}

\begin{example}
 Let $Y$ be the geometric realization of a $1$-dimensional simplicial
 complex of the shape of Y as is shown in Figure \ref{figure:Y}. 
 \begin{figure}[ht]
  \begin{center}
   \begin{tikzpicture}
    \draw (-0.5,0.87) -- (0,0);
    \draw (0.5,0.87) -- (0,0);
    \draw (0,-1) -- (0,0);
    \draw [fill] (-0.5,0.87) circle (2pt);
    \draw (-0.7,1.3) node {$e^{0}_{1}$};
    \draw [fill] (0.5,0.87) circle (2pt);
    \draw (0.7,1.3) node {$e^{0}_{2}$};
    \draw [fill] (0,0) circle (2pt);
    \draw [fill] (0,-1) circle (2pt);
    \draw (0,-1.4) node {$e^{0}_{0}$};
   \end{tikzpicture}
  \end{center}
  \caption{A stellar stratification on ``Y''}
  \label{figure:Y}
 \end{figure}
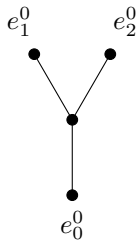
 Let $e^{0}_{0}, e^{0}_{1}, e^{0}_{2}$ be the three outer vertices.
 Denote the complement $Y\setminus \{e^{0}_{0}, e^{0}_{1}, e^{0}_{2}\}$ 
 by $e^1$. Then the decomposition 
 \[
  Y = e^{0}_{0} \cup e^{0}_{1} \cup e^{0}_{2} \cup e^{1}
 \]
 is a stellar stratification. The stellar structure on $e^{1}$ is given
 by the identity map $Y\to \overline{e^{1}}$.
\end{example}

\begin{definition}
 Let $(X,\pi_X,\Phi_X)$ and $(Y,\pi_Y,\Phi_Y)$ be stellar stratified
 spaces. A \emph{morphism of stellar stratified spaces}
 from $(X,\pi_X,\Phi_X)$ to $(Y,\pi_Y,\Phi_Y)$ consists of 
 \begin{itemize}
  \item a morphism $\bm{f} : (X,\pi_X) \to (Y,\pi_Y)$ of stratified
	spaces, and
  \item a family of maps
	$f_{\lambda}:D^{X}_{\lambda}\to D^{Y}_{P(f)(\lambda)}$
	indexed by $P(X)$ making the diagrams 
	\[
	 \begin{diagram}
	  \node{X} \arrow{e,t}{f} \node{Y} \\
	  \node{D^{X}_{\lambda}} \arrow{n,l}{\varphi_{\lambda}}
	  \arrow{e,b}{f_{\lambda}}
	  \node{D^{Y}_{P(f)(\lambda)}} 
	  \arrow{n,r}{\psi_{P(f)(\lambda)}} 
	 \end{diagram}
	\]
	commutative, where $\varphi_{\lambda}$ and
	$\psi_{\underline{f}(\lambda)}$ are stellar structures in $X$
	and $Y$, respectively. 
 \end{itemize}
 The category of stellar stratified spaces is denoted by
 $\SSS$. 
\end{definition}

The aim of \cite{1609.04500} was to find a structure on a stellar
stratified space from which the homotopy type can be recovered.
The notion of cylindrical normality was introduced for
this purpose. 
Let us recall the definition.

\begin{definition}
 \label{definition:cylindrical_normality}
 Let $\pi:X\to P(X)$ be a normal stellar stratified space whose stellar 
 structure is given by
 $\{\varphi_{\lambda}:D_{\lambda}\to \overline{e_{\lambda}}\}$.
 A \emph{cylindrical structure} on this stellar stratified space
 consists of
 \begin{itemize}
   \item a space $P_{\mu,\lambda}$ and a strict morphism of
	stratified spaces 
	\[
	 b_{\mu,\lambda} : P_{\mu,\lambda}\times D_{\mu}
	 \rarrow{}
	 \partial D_{\lambda}
	\]
	 for each pair of strata
	 $e_{\mu} \subset \partial e_{\lambda}$,
	 where each $P_{\mu,\lambda}$ is regarded as a stratified space
	 with a single stratum,	 and
  \item a map
	\[
	 c_{\lambda_0,\lambda_1,\lambda_2} :
	 P_{\lambda_1,\lambda_2}\times P_{\lambda_0,\lambda_1} 
	 \longrightarrow P_{\lambda_0,\lambda_2}
	\]
	 for each sequence $\overline{e_{\lambda_0}} \subset
	\overline{e_{\lambda_1}} \subset \overline{e_{\lambda_2}}$
 \end{itemize}
 satisfying the following conditions:
 \begin{enumerate}
  \item The restriction of $b_{\mu,\lambda}$ to
	$P_{\mu,\lambda}\times \inte(D_{\mu})$ is a homeomorphism
	onto its image. 
  \item The following three types of diagrams are commutative. 
	\[
	 \begin{diagram}
	  \node{D_{\lambda}} \arrow{e,t}{\varphi_{\lambda}}
	  \node{X} \\
	  \node{P_{\mu,\lambda}\times D_{\mu}}
	  \arrow{n,l}{b_{\mu,\lambda}} 
	  \arrow{e,t}{\pr_2} 
	  \node{D_{\mu}.} \arrow{n,r}{\varphi_{\mu}}
	 \end{diagram}
	\]
	\[
	 \begin{diagram}
	  \node{P_{\lambda_1,\lambda_2}\times
	  P_{\lambda_0,\lambda_1}\times D_{\lambda_0}}
	  \arrow{s,l}{c_{\lambda_0,\lambda_1,\lambda_2}\times 1} 
	  \arrow{e,t}{1\times b_{\lambda_0,\lambda_1}}
	  \node{P_{\lambda_1,\lambda_2}\times D_{\lambda_1}}
	  \arrow{s,r}{b_{\lambda_1,\lambda_2}} \\ 
	  \node{P_{\lambda_0,\lambda_2}\times D_{\lambda_0}}
	  \arrow{e,b}{b_{\lambda_0,\lambda_2}} \node{D_{\lambda_2}} 
	 \end{diagram}
	\]
	\[
	 \begin{diagram}
	  \node{P_{\lambda_2,\lambda_3}\times
	  P_{\lambda_1,\lambda_2}\times P_{\lambda_0,\lambda_1}}
	  \arrow{e,t}{c_{\lambda_{1},\lambda_{2},\lambda_{3}}\times 1}
	  \arrow{s,l}{1\times c_{\lambda_{0},\lambda_{1},\lambda_{2}}}  
	  \node{P_{\lambda_1,\lambda_3}\times P_{\lambda_0,\lambda_1}}
	  \arrow{s,r}{c_{\lambda_{0},\lambda_{1},\lambda_{3}}} \\ 
	  \node{P_{\lambda_2,\lambda_3}\times P_{\lambda_0,\lambda_2}}
	  \arrow{e,b}{c_{\lambda_{0},\lambda_{2},\lambda_{3}}} 
	  \node{P_{\lambda_0,\lambda_3}.} 
	 \end{diagram}
	\]
  \item We have 
	\[
	 D_{\lambda} = \bigcup_{e_{\mu}\subset \partial
	e_{\lambda}}  b_{\mu,\lambda}(P_{\mu,\lambda}\times
	\inte(D_{\mu})) 
	\]
	as a stratified space.
 \end{enumerate}

 The space $P_{\mu,\lambda}$ is called the \emph{parameter space} for the
 inclusion $e_{\mu}\subset \overline{e_{\lambda}}$. When $\mu=\lambda$,
 we define $P_{\lambda,\lambda}$ to be a single point.
 A stellar stratified space equipped with a cylindrical
 structure is 
 called a \emph{cylindrically normal stellar stratified space} (CNSSS,
 for short).
\end{definition}

\begin{definition}
 \label{definition:cylindrical_face_category}
 For a CNSSS $X$, define a category $C(X)$ as follows.
 Objects are strata of $X$.
 For each pair $e_{\mu} \subset \overline{e_{\lambda}}$, define
 \[
  C(X)(e_{\mu},e_{\lambda}) = P_{\mu,\lambda}.
 \]
 The composition of morphisms is given by
 \[
  c_{\lambda_0,\lambda_1,\lambda_2} : P_{\lambda_1,\lambda_2}\times
 P_{\lambda_0,\lambda_1} \longrightarrow P_{\lambda_0,\lambda_2}.
 \]
 The category $C(X)$ is called the \emph{face
 category} of $X$.
\end{definition}

The following fact is obvious from the definition.

\begin{lemma}
 \label{face_category_of_CNCSS}
 For any CNSSS $X$,
 its face category $C(X)$ is an acyclic topological category whose
 underlying poset is $P(X)$.
\end{lemma}

\begin{definition}
 Let $(X,\pi_X,\Phi_X)$ and $(Y,\pi_Y,\Phi_Y)$ be CNSSSs
 with cylindrical structures given by 
 $\{b^X_{\mu,\lambda} : P^X_{\mu,\lambda}\times D^{X}_{\mu}\to D^{X}_{\lambda}\}$
 and
 $\{b^Y_{\alpha,\beta} : P^Y_{\alpha,\beta}\times D^{Y}_{\alpha}\to
 D^{Y}_{\beta}\}$, respectively.

 A \emph{morphism of CNSSSs}
 from $(X,\pi_X,\Phi_X)$ to $(Y,\pi_Y,\Phi_Y)$ is a morphism of stellar
 stratified spaces 
 \[
 \bm{f} : (X,\pi_X,\Phi_X) \longrightarrow
 (Y,\pi_Y,\Phi_Y) 
 \]
 together with maps
 $f_{\mu,\lambda} : P_{\mu,\lambda}^{X} \to
 P^Y_{P(f)(\mu),P(f)(\lambda)}$ making the diagram
 \[
  \begin{diagram}
   \node{P^X_{\mu,\lambda}\times D_{\mu}}
   \arrow{e,t}{f_{\mu,\lambda}\times f_{\mu}}
   \arrow{s,l}{b^X_{\mu,\lambda}} 
   \node{P^Y_{P(f)(\mu),P(f)(\lambda)}\times
   D_{P(f)(\mu)}}
   \arrow{s,r}{b^Y_{\underline{f}(\mu),P(f)(\lambda)}} \\  
   \node{D_{\lambda}} \arrow{e,b}{f_{\lambda}}
   \node{D_{P(f)(\lambda)}} 
  \end{diagram}
 \]
 commutative.
 
 The category of CNSSSs is denoted by $\CNSSS$.
 The full subcategories of cylindrically normal CW stellar complexes and
 of cylindrically normal cellular stratified spaces are
 denoted by $\category{CNCW}$ and $\category{CNCSS}$, respectively. 
\end{definition}

One of main results of \cite{1609.04500} is the following.

\begin{theorem}[Theorem 5.16 of \cite{1609.04500}]
 \label{Salvetti}
 For a CW cylindrically normal cellular stratified space $X$, there
 exists a natural embedding $i_{X}:BC(X)\hookrightarrow X$ which is a
 homeomorphism when $X$ is a CW complex. 
\end{theorem}

The embedding $i_{X}$ is, in fact, constructed for
``cylindrically normal stellar stratified spaces'' 
in the sense of \cite{1609.04500}. They differ from CNSSS in this paper
by the requirement that the domain $D_{\lambda}$ of the stellar
structure $\varphi_{\lambda}:D_{\lambda}\to \overline{e_{\lambda}}$ of
each stratum is embedded in a disk. This embeddability condition is not
used in the construction of $i_{X}$.

\begin{corollary}
 The embedding in Theorem \ref{Salvetti} can be extended to CW
 CNSSSs. When $X$ is a cylindrically normal stellar complex, the
 embedding $i_{X}:BC(X)\to X$ is a homeomorphism.
\end{corollary}

\section{Stellar Stratifications on Classifying Spaces of Acyclic Categories}
\label{stellar_stratification_on_BC}

Note that the space of objects in the face category $C(X)$ of a CNSSS
$X$ has the discrete topology. Let us call such a topological category a
\emph{top-enriched category}.
In the rest of this paper, we restrict our attention to acyclic
top-enriched categories.
We introduce two stellar stratifications on $BC$ for such a
category $C$. 

\subsection{Stable and Unstable Stratifications}
\label{stable_and_unstable}

Let $C$ be an acyclic top-enriched category.
We first need to define stratifications on $BC$.
An obvious choice is the simplicial stratification
\[
 \pi_{\nN(C)} : BC \rarrow{} \coprod_{n\ge 0}\nN_{n}(C) 
\]
in Example \ref{example:simplicial_complex}, since $BC$ is 
homeomorphic to the geometric realization of a $\Delta$-space
$\nN(C)$.
Unfortunately this stratification is too fine for our purpose.

\begin{definition}
 \label{definition:unstable_stratification}
 The composition
 \[
 BC \rarrow{\pi_{\nN(C)}} \coprod_{n\ge 0}\nN_{n}(C) \rarrow{t}
 C_{0} 
 \]
 is denoted by $\pi_{C}$. It is easy to see that this is a
 stratification when $C_{0}$ is regarded as the poset $P(C)$ associated
 with $C$.
 This is called the \emph{unstable stratification} on $BC$.

 We also have a dual stratification
 \[
 \pi_{C^{\op}} : BC \rarrow{\pi_{\nN(C)}}
 \coprod_{n\ge 0}\nN_{n}(C) \rarrow{s} C_{0},
 \]
 which should be called the \emph{stable stratification}.
\end{definition}

\begin{remark}
 We regard $BC$ as a stratified space by the unstable stratification.
 The stable stratification will be used in \S\ref{exit-path}. 
 These terminologies are borrowed from the
 ``classifying space approach'' to Morse theory
 \cite{Cohen-Jones-SegalMorse,1612.08429}.  
\end{remark}

\begin{example}
 Consider the acyclic top-enriched category $C$ with two objects $x$ and
 $y$ and $C(x,y)=S^{n-1}$. The classifying space $BC$ is a quotient of
 \[
 C_{0}\times\DDelta^{0} \amalg
 (C_{1}\setminus\{1_{x},1_{y}\})\times\DDelta^{1} \cong 
 C_{0}\amalg C(x,y)\times [0,1].
 \]
 The relation is defined by identifying $C(x,y)\times\{0\}$ and
 $C(x,y)\times\{1\}$ with $x$ and $y$, respectively. Thus it is 
 a suspension of $S^{n-1}$.

 The simplicial stratification
 $\pi_{\nN(C)}:BC=\Sigma(S^{n-1}) \to C_{0}=\{x,y\}$
 is given by 
 \[
 \pi_{\nN(C)}([u,t]) =
 \begin{cases}
  x, & t=0, \\
  u, & 0<t<1, \\
  y, & t=1,
 \end{cases}
 \]
 while the unstable stratification $\pi_{C}$ is given by
 \[
 \pi_{C}([u,t]) =
 \begin{cases}
  x, & t<1 \\
  y, & t=1.
 \end{cases}
 \]
 The stratum indexed by $x$ can be identified with the space of flows
 going out of $x$. This example justify the terminology.
\end{example}

\begin{example}
 \label{example:max_stratification_on_simplex}
 Let $P=[2]=\{0<1<2\}$. The classifying space $BP$ is homeomorphic
 to the standard $2$-simplex $\DDelta^{2}$. Under this identification,
 the unstable stratification $\pi_{[2]}:\DDelta^{2}\to [n]$ is given by
 \[
  \pi_{[2]}(t_{0},t_{1},t_{2})= \max\set{i}{ t_{i}\neq 0}.
 \]
 Let $e_{i}=\pi_{[2]}^{-1}(i)$ for $i\in[2]$. This stratification on
 $\DDelta^2$ is given as in Figure
 \ref{figulre:max_stratification_on_2-simplex}. 

 \begin{figure}[ht]
  \begin{center}
   \begin{tikzpicture}
    \draw [fill,lightgray] (0,0) -- (2,0) -- (1,1.73) -- (0,0);
    \draw (0,0) -- (2,0) -- (1,1.73) -- (0,0);
    \draw [fill] (0,0) circle (2pt);
    \draw [fill] (2,0) circle (2pt);
    \draw [fill] (1,1.73) circle (2pt);
    \draw (1,-0.3) node {$\DDelta^2$};
    
    \draw (2.5,1) node {$=$};

    \draw [fill] (3,0) circle (2pt);
    \draw (3,-0.3) node {$e_{0}$};

    \draw (3.5,0.5) node {$\cup$};

    \draw (4,0) -- (6,0);
    \draw [fill] (4,0) circle (2pt);
    \draw [fill,white] (4,0) circle (1pt);
    \draw [fill] (6,0) circle (2pt);
    \draw (5,-0.3) node {$e_{1}$};

    \draw (6.5,0.5) node {$\cup$};

    \draw [fill,lightgray] (7,0) -- (8,1.73) -- (9,0) -- (7,0);
    \draw (7,0) -- (8,1.73) -- (9,0);
    \draw [dotted] (7,0) -- (9,0);
    \draw [fill] (7,0) circle (2pt);
    \draw [fill,white] (7,0) circle (1pt);
    \draw [fill] (9,0) circle (2pt);
    \draw [fill,white] (9,0) circle (1pt);
    \draw [fill] (8,1.73) circle (2pt);
    \draw (8,-0.3) node {$e_{2}$};    
   \end{tikzpicture}
  \end{center}
  \caption{The unstable stratification on $B[2]$}
  \label{figulre:max_stratification_on_2-simplex}
 \end{figure}
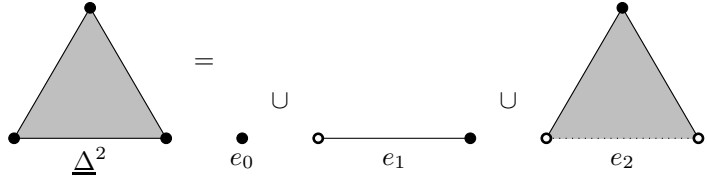

 The inclusions $e_{0}\subset \overline{e_{1}}$ and
 $e_{1}\subset\overline{e_{2}}$ imply that the poset $[2]$ can be
 recovered from this stratification.
 
 More generally, the unstable stratification on $B[n]=\DDelta^{n}$ is
 given by
 \[
  \pi_{[n]}(t_{0},\ldots,t_{n}) = \max\set{i}{t_{i}\neq 0}
 \]
 This is the stratification appeared as Example 2.10 in \cite{1609.04500}.
 This stratification is used in the definition of the exit-path
 $\infty$-category. See \S\ref{exit-path}.
\end{example}

Our next task is to define a stellar structure on the unstable
stratification. This is done by using comma
categories.  

\begin{definition}
 \label{upper_and_lower_star}
 Let $C$ be an acyclic top-enriched category and $x$ an object of
 $C$. The nondegenerate nerve $\nN(C\downarrow x)$ of the
 comma category $C\downarrow x$ is denoted by $\St_{\le x}(C)$ and is
 called the \emph{lower star} of $x$ in $C$.

 The full subcategory of $C\downarrow x$ consisting of
 $(C\downarrow x)_{0}\setminus\{1_x\}$ is denoted by $C_{<x}$. 
 The nondegenerate nerve $\nN(C_{<x})$ is 
 denoted by $\Lk_{<x}(C)$ and is called the \emph{lower link} of $x$
 in $C$. 

 The functor induced by the source map in $C$ is denoted by
 \[
  s_x : C_{<x}\subset C\downarrow x \rarrow{} C.
 \]
 The induced map of $\Delta$-spaces is also denoted by
 \[
 s_x : \Lk_{<x}(C) \subset \St_{\le x}(C) \rarrow{}
 \nN(C). 
 \]

 Dually, we define $\St_{\ge x}(C)$, $C_{>x}$, $\Lk_{>x}(C)$ and
 $t_{x}: C_{>x} \subset x\downarrow C\to C$ by using
 $x\downarrow C$. The map induced by the functor $t_{x}$ is denoted by
 $t_{x}: \Lk_{>x}(C)\subset \St_{\ge x}(C)\to \nN(C)$.
\end{definition}

 It is straightforward to verify that $C\downarrow x$ and
 $x\downarrow C$ are acyclic when $C$ is.

We have the following description of the lower link.

\begin{lemma}
 \label{link}
 For an acyclic top-enriched category $C$ and an object $x\in C_0$,
 we have
 \[
  \Lk_{<x}(C)_k \cong 
 \begin{cases}
  \displaystyle \coprod_{x\neq y} C(y,x), & k=0 \\
  \set{\bm{u}\in \nN_{k+1}(C)}{t(\bm{u})=x}, & k>0.
 \end{cases} 
 \]
\end{lemma}

\begin{proof}
 By inspection.
\end{proof}

\begin{definition}
 For $x\in C_{0}$, define
 \begin{align*}
  D_{x} & = \|\St_{\le x}(C)\|=B(C\downarrow x) \\
  \partial D_{x} & = \|\Lk_{<x}(C)\| = BC_{<x} \\
  D_{x}^{\op} & = \|\St_{\ge x}(C)\| = B(x\downarrow C) \\
  \partial D_{x}^{\op} & =  \|\Lk_{>x}(C)\| = BC_{>x}.
 \end{align*}
 The complements $D_{x}\setminus\partial D_{x}$ and
 $D_{x}^{\op}\setminus\partial D_{x}^{\op}$ are denoted by
 $D_{x}^{\circ}$ and $D_{x}^{\op,\circ}$, respectively.
 The maps induced by the source and the target maps are denoted by
 $s_{x} : D_{x} \to BC$ and $t_{x}: D_{x}^{\op}\to BC$, respectively.
\end{definition}

The geometric realization $\|\Lk_{<x}(C)\|$ has a stratification based
on the $\Delta$-space structure.
Lemma 7.12 of \cite{1609.04500} says that we have a homeomorphism
\[
 D_{x}^{\op}=\|\St_{\ge x}(C)\| \cong \cone(\|\Lk_{>x}(C)\|)=
 \{1_{x}\}\star \partial D_{x}^{\op}.
\]
The following is a dual.

\begin{lemma}
 \label{star_is_cone}
 For an acyclic top-enriched category $C$ and an object $x\in C_{0}$, we 
 have a homeomorphism
 \[
  D_{x} = \|\St_{\le x}(C)\| \cong \cone(\|\Lk_{<x}(C)\|) = 
 \partial D_{x}\star\{1_{x}\}.
 \]
 In particular, $D_{x}^{\circ}$ is an open cone on $\partial D_{x}$.
\end{lemma}

In order to use explicit descriptions of these homeomorphisms, let us
sketch a proof. 

\begin{proof} 
 For $x\in C_{0}$, defined a map
 $h_{x} : D_{x} \to  BC_{<x} \star\{1_{x}\}$
 as follows.
 For $[(\bm{u},\bm{a})]\in D_{x}=\|\nN(x\downarrow C)\|$, 
 choose a representative
 $(\bm{u},\bm{a})\in \nN_{k}(C\downarrow x)\times\DDelta^{k}$ such that
 $\bm{a}$ is in the interior of $\DDelta^{k}$. The $k$-chain
 $\bm{u}$ in $C\downarrow x$ can be regarded as a $(k+1)$-chain in $C$
 of the following form
 \[
  \bm{u}: x_{0} \rarrow{u_1} x_{1} \to \cdots \to
 x_{k-1} \rarrow{u_{k}} x_{k} \rarrow{u_{k+1}} x.
 \]
 Note that $u_{1},\ldots,u_{k}$ are not identity morphisms, but
 $u_{k+1}$ can be. When $u_{k+1}$ is not an identity morphism, 
 $[\bm{u},\bm{a}]$ defines an element of $BC_{<x}$. Define
 \[
  h_{x}([\bm{u},\bm{a}]) = 1[\bm{u},\bm{a}]+ 01_{x}.
 \]
 When $u_{k+1}=1_{x}$, use the standard identification
 $\DDelta^{k}\cong \DDelta^{k-1}\star \{\bm{e}_{k+1}\}$ to denote
 $\bm{a} = (1-t)\bm{a}'+t\bm{e}_{k+1}$, where $\bm{e}_{i}$ is the $i$-th
 vertex of $\DDelta^{k}$. And define
 \[
 h_{x}([\bm{u},\bm{t}]) =
 (1-t)[(u_{1},u_{2},\ldots,u_{k+1}\circ u_{k}),\bm{a}']
 +t1_{x}.
 \]
 Here we regard $(u_{1},u_{2},\ldots,u_{k+1}\circ u_{k})$ as an element
 of $\nN_{k-1}(C_{<x})$ under the identification of Lemma \ref{link}.

 A map $h_{x}^{\op}: D_{x}^{\op}\to \{1_{x}\}\star BC_{>x}$ is
 defined by reversing arrows.
 It is straightforward to verify that these maps are homeomorphisms.
\end{proof}

In order to describe the image of $D^{\circ}_{x}$ under $s_{x}$, 
we need the following generalization of the dual of Lemma 7.15 of
\cite{1609.04500}. 

\begin{lemma}
 \label{stratum_in_BC}
 Let $C$ be an acyclic top-enriched category.
 Then, for each $x\in C_{0}$,
 we have
 \[
 \pi_{C}^{-1}(x) = s_{x}(D_{x}^{\circ}) \text{ and }
  \overline{\pi_{C}^{-1}(x)} = s_{x}(D_{x}).
 \]
 Furthermore
 \begin{equation}
  s_{x}(D_{x}^{\circ}) =
   p_{C}\left(
	 \coprod_{k} \coprod_{x_{0}<\cdots<x_{k-1}<x}
	 C(x_{k-1},x)\times \cdots\times  C(x_{0},x_{1})\times
	 \left(\DDelta^k \setminus d^{k}(\DDelta^{k-1})\right)
	\right),
   \label{equation:image_of_sx}
 \end{equation}
 where $p_{C}:\coprod_{k}\nN_{k}(C)\times\DDelta^k\to BC$ is the
 projection.
\end{lemma}

\begin{proof}
 Let us first show (\ref{equation:image_of_sx}).
 An element of $D_{x}^{\circ}$ can be represented by a pair
 $(\bm{u},\bm{t})$ of $\bm{u}\in \nN_{k}(C_{<x})$ and
 $\bm{t}\in \DDelta^k$, where $\bm{u}$ is a sequence of morphisms
 $v_{0}\rarrow{u_1} v_{1}\to \cdots \to v_{k-1}\rarrow{u_{k}} v_{k}$
 in $C_{<x}$ with $v_{i}\neq 1_{x}$ for all $i$ and $u_{k}=1_{x}$.
 The only possibility for such an element to be equivalent to an element
 of $\coprod_{k}\nN_{k}(C_{<x})\times\DDelta^k$ is that
 $\bm{t}\in d^k(\DDelta^{k-1})$. And we obtain
 (\ref{equation:image_of_sx}). 

 By definition, the composition $\pi_{C}\circ s_{x}$ is a constant map
 onto $x$ when restricted to $D^{\circ}_{x}$ and thus
 $s_{x}(D^{\circ}_{x})\subset \pi_{C}^{-1}(x)$. 

 Suppose $\pi_{C}([\bm{u},\bm{t}])=x$ for
 $\bm{u}=(u_{k},\ldots,u_{1})\in\nN_{k}(C)$ and
 $\bm{t}\in \inte\DDelta^{k}$. 
 Then $t(u_{k})=x$ and the sequence $(u_{k},\ldots,u_{1})$ can be
 regarded as an element of
 $\nN_{k}(C\downarrow x)=\St_{\le x}(C)_{k}$ as follows
 \begin{equation}
  \begin{diagram}
   \node{x_{0}} \arrow{e,t}{u_1}
   \arrow{ese,b}{u_{k}\circ\cdots\circ u_{1}} \node{\cdots}
   \arrow{e,t}{u_k} \node{x} \arrow{s,r}{1_{x}} \\
   \node{} \node{} \node{x.}
  \end{diagram}
 \label{diagram:adding_1_at_the_end}
 \end{equation}
 Since $\bm{t}\in\inte\Delta^{k}$, the pair $(\bm{u},\bm{t})$ represents
 and element of $D_{x}^{\circ}$ whose image under $s_{x}$ is
 $[\bm{u},\bm{t}]$. Thus we have $s_{x}(D_{x}^{\circ})=\pi_{C}^{-1}(x)$.

 By taking the closure, we have
 $\overline{s_{x}(D_{x}^{\circ})}=\overline{\pi_{C}^{-1}(x)}$.
 Since $C_{0}$ is discrete, the topology of $BC$ is given by the weak
 topology defined by the covering
 \[
  \{p_{C}(C(x_{k-1},x_{k})\times\cdots\times C(x_{0},x_{1})\times
 \Delta^k)\}_{k\ge 0, x_{0}\le\cdots\le x_{k}}.
 \]
 Under the description of (\ref{equation:image_of_sx}), the closure of
 $s_{x}(D_{x}^{\circ})$ is given by adding
 $C(x_{k-1},x)\times \cdots\times C(x_{0},x_{1})\times d^{k}(\DDelta^k)$.
 And we have $\overline{\pi_{C}^{-1}(x)}=s_{x}(D_{x})$.
\end{proof}

Note that we used the fact that $C_{0}$ has the discrete topology.

In order for $s_{x}$ to be a stellar structure, we need to impose a
finiteness condition on $C$.

\begin{definition}
 An acyclic top-enriched category $C$ is called \emph{locally finite} if
 $P(C)_{<x}=\set{y\in C_{0}}{y<x}$ is finite.
\end{definition}

\begin{lemma}
 Let $C$ be an acyclic top-enriched category. If $C$ is locally finite
 and the morphism space $C(x,y)$ is compact Hausdorff for each pair
 $x,y\in C_{0}$, then $s_{x}: D_{x}\to BC$ is a stellar structure on
 $\pi_{C}^{-1}(x)$ for each $x\in C_{0}$.
\end{lemma}

\begin{proof}
 By Lemma \ref{star_is_cone}, $D_{x}$ is a cone on $\partial D_{x}$.
 By Lemma \ref{stratum_in_BC}, the image of the map $s_{x}$ is
 $\overline{\pi_{C}^{-1}(x)}$.
 By the compactness of each $C(x,y)$ and the finiteness of $P(C)({<x})$,
 $D_{x}$ is compact. By a result of de Seguins Pazzis
 \cite{1005.2666}, $BC$ is Hausdorff. Hence
 $s_{x}:D_{x}\to\overline{\pi_{C}^{-1}(x)}$ is a quotient 
 map.

 The fact that the
 restriction of $s_{x}$ to $D_{x}^{\circ}$ is 
 a homeomorphism onto $\pi_{C}^{-1}(x)$ follows from the description
 (\ref{equation:image_of_sx}). In fact, the inverse to
 $s_{x}|_{D_{x}^{\circ}}$ is given by assigning
 (\ref{diagram:adding_1_at_the_end}) to
 $(u_{k},\ldots,u_{1})\in\overline{N}_{k}(C)$ with $t(u_{k})=x$.
\end{proof}

The stratification we have defined on $BC$ fits well into the face
category of CNSSS.

\begin{proposition}
 \label{stratification_on_BC(X)}
 For a CW CNSSS $X$, the embedding 
 \[
 i_{X}: BC(X) \hookrightarrow X
 \]
 in Theorem \ref{Salvetti} is a morphism of stratified spaces
 when $BC(X)$ is equipped with the unstable stratification.
 If each parameter space of $X$ is compact Hausdorff, $i_{X}$ is a
 morphism of stellar stratified spaces.

 In particular, if $X$ is a CW stellar complex, $i_{X}$ is an
 isomorphism of stellar stratified spaces. 
\end{proposition}

\begin{proof}
 This follows immediately from the explicit definitions of $i_{X}$ and the
 unstable stratification on $BC(X)$. Note that the closure finiteness
 implies that $C(X)$ is locally finite.
\end{proof}

The next task is to find a cylindrical structure, in the sense of
Definition \ref{definition:cylindrical_normality}, on the stellar
stratification on $BC$ we have constructed.

\begin{definition}
 Let $C$ be an acyclic topological category.
 For $x,y\in C_{0}$ with $x<y$, a morphism $u:x\to y$ induces a functor 
 \[
  u\circ (-) : C\downarrow x \rarrow{} C\downarrow y.
 \]
 The induced map on the classifying spaces is denoted by
 \[
  b_{x,y} : C(x,y)\times D_{x} \rarrow{} D_{y}.
 \]
\end{definition}

We are ready to prove Theorem \ref{theorem:stratification}.

\begin{proof}[Proof of Theorem \ref{theorem:stratification}]
 It remains to show that maps
 $b_{x,y}: C(x,y)\times D_{x}^{\circ}\to D_{y}$ define
 a cylindrical structure on the unstable stratification on $BC$.

 The fact that each $b_{x,y}:C(x,y)\times D_{x}^{\circ}\to D_{y}$ is an
 embedding follows from the acyclicity of $C$.
 The associativity of compositions of morphisms in $C$ implies that the
 commutativity of three diagrams in Definition
 \ref{definition:cylindrical_normality}.
 By definition, $e_{x}\subset\partial e_{y}$ if and only if $x\neq y$
 and $C(x,y)\neq\emptyset$ and we have
 \[
 D_{y}= \bigcup_{e_{x}\subset\partial e_{y}}
 b_{x,y}(C(x,y)\times D_{x}^{\circ}).
 \]
\end{proof}

Theorem \ref{theorem:equivalence} is a corollary to the above argument.

\begin{proof}[Proof of Theorem \ref{theorem:equivalence}]
 By Lemma \ref{star_is_cone}, the structure of CNSSS on $BC$ constructed
 in Theorem \ref{theorem:stratification} is actually a stellar complex.
 It remains to verify that it is CW.
 The closure finiteness follows from the locally finiteness of $C$.
 For the weak topology, consider the commutative diagram
 \[
  \begin{diagram}
   \node{\coprod_{k\ge 0} \nN_{k}(C)\times\DDelta^{k}} \arrow{s,l}{p_{C}}
   \arrow{e,=} \node{\coprod_{x\in C_{0}}\coprod_{k}
   t_{k}^{-1}(x)\times\DDelta^{k}} \arrow{s,r}{\coprod_{x\in C_{0}}
   p_{D_{x}}} \\  
   \node{BC} \node{\coprod_{x\in C_{0}} D_{x},}
   \arrow{w,b}{\bigcup_{x\in C_{0}} s_{x}} 
  \end{diagram}
 \]
 where $t_{k}:\nN_{k}(C)\to C_{0}$ is given by the target map.
 Since both $p_{C}$ and $\coprod_{x\in C_{0}} p_{D_{x}}$ are quotient
 maps, it follows that $\bigcup_{x\in C_{0}} s_{x}$ is a quotient map.
\end{proof}
\subsection{The Exit-Path Category}
\label{exit-path}

In this section, we prove Theorem \ref{theorem:conically_stratified}.
Let us first recall the definition of the exit-path category.

\begin{definition}
 \label{definition:exit-path_category}
 For a stratified space $\pi:X\to P(X)$, define 
 \begin{align*}
 \Exit(X) & = \set{\sigma\in
  \Sing(X)}{\sigma:\DDelta^n\to X \text{ stratification
  preserving}} \\
  & = \set{\sigma:\DDelta^n\to X}{\hspace*{10pt}
   \begin{minipage}{90pt}
    $\begin{diagram}
      \node{\DDelta^n} \arrow{s,l}{\pi_{[n]}} \arrow{e,t}{\sigma}
      \node{X} \arrow{s,r}{\pi_{X}} \\
      \node{[n]} \arrow{e,b}{\exists} \node{P(X),}
     \end{diagram}$
   \end{minipage}
   }
 \end{align*}
 where $\Sing(X)$ is the singular simplicial set of $X$ and $\pi_{[n]}$
 is the stratification in Example
 \ref{example:max_stratification_on_simplex}. 
 When $\Exit(X)$ is a quasi-category, it is called the
 \emph{exit-path $\infty$-category of $X$}.
\end{definition}

The following useful criterion can be found in Lurie's book
\cite{LurieHigherAlgebra}.

\begin{definition}[Definition A.5.5 of \cite{LurieHigherAlgebra}]
 \label{definition:conically_stratified}
 A stratified space $\pi: X\to P(X)$ is called \emph{conically stratified}
 if, for each $x \in X$, there exists a $P(X)_{>\pi(x)}$-stratified
 space $Y$, a topological space $Z$, and an open embedding
 $Z\times\cone^{\circ}(Y)\to X$ of $P(X)$-stratified space whose image
 contains $x$, where $P(X)_{>\pi(x)}$ is the full subposet of $P(X)$
 consisting of elements $\lambda$ with $\lambda>\pi(x)$. 
\end{definition}

\begin{theorem}[Theorem A.6.4 (1) of \cite{LurieHigherAlgebra}]
 \label{theorem:conical_implies_inner_fibration}
 If a stratified space $\pi: X\to P(X)$ is conically stratified, the
 the map $p_{\pi}:\Exit(X)\to N(P(X))$ induced by $\pi$ is an inner
 fibration.
 In particular, $\Exit(X)$ is a quasi-category.
\end{theorem}

Thanks to this theorem, it suffices to show that the unstable
stratification on $BC$ is conically stratified in order to prove Theorem
\ref{theorem:conically_stratified}. 
This is done by combining with the dual stratification,
i.e.\ the stable stratification on $BC$.
Namely there exists a map
\begin{equation}
 c_{x} : D_{x}\times D_{x}^{\op} \rarrow{} BC
 \label{conical_nbh} 
\end{equation}
for each object $x$ in $C$ such that the restrictions
$c_{x}|_{D_{x}\times \{1_{x}\}}$ and 
$c_{x}|_{\{1_{x}\}\times D_{x}^{\op}}$ coincide with $s_{x}$ and
$t_{x}$, respectively. 
The map $c_{x}$ is defined by the composition
\[
 D_{x}\times D_{x}^{\op} \rarrow{1_{D_{x}}\times h_{x}^{\op}}
 D_{x}\times (\{1_{x}\}\star \partial 
 D_{x}^{\op}) \rarrow{j_{x}} D_{x} \star \partial D_{x}^{\op}
 \rarrow{n_{x}} BC, 
\]
where $h_{x}^{\op}$ is the map defined in the proof of Lemma
\ref{star_is_cone} and the map $j_{x}$ is given by
$j_{x}(\bm{a},(1-t)1_{x}+t\bm{b})=(1-t)\bm{a}+t\bm{b}$.
Let us define $n_{x}$.

\begin{definition}
 For $[\bm{u},\bm{s}]\in D_{x}$, choose a representative with
 $\bm{u}\in \nN_{p}(C\downarrow x)$ and $\bm{s}\in \inte\DDelta^{p}$.
 The nondegenerate $p$-chain
 \[
 \bm{u}:y_{0}\rarrow{u_1} y_{1}\rarrow{u_2} \cdots \rarrow{u_{p}}
 y_{p}\rarrow{u_{p+1}} x
 \]
 in $C\downarrow x$ is regarded as a $(p+1)$-chain in $C$. Note that
 this may degenerate in $N_{p+1}(C)$.
 For $[\bm{v},\bm{t}]\in \partial D_{x}^{\op}$, choose a
 representative with $\bm{v}\in \nN_{q}(C_{>x})$ and
 $\bm{t}\in\inte\DDelta^{q}$. The nondegenerate $q$-chain
 \[
  \bm{v}: x\rarrow{v_{0}} z_{0} \rarrow{v_{1}} \cdots \rarrow{v_{q}} z_{q}
 \]
 in $x\downarrow C$ is also nondegenerate as an element of $N_{q+1}(C)$,
 since $v_{0}\neq 1_{x}$.

 Now define
 \[
 n_{x}((1-r)[\bm{u},\bm{s}]+r[\bm{v},\bm{t}]) =
 \begin{cases}
  [(v_{q},\ldots,v_{1},v_{0}\circ u_{p+1},u_{p},\ldots,u_{1}),
  (1-r)\bm{s}+r\bm{t}], & r>0 \\
  [(u_{p},\ldots,u_{1}),\bm{s}] =s_{x}([\bm{u},\bm{s}]), & r=0
 \end{cases} 
 \]
 under the identification of
 $\DDelta^{p}\star\DDelta^{q}\cong \DDelta^{p+q+1}$.
\end{definition}

\begin{lemma}
 \label{n_is_embedding}
 The map $n_{x}$ is an embedding onto its image.
\end{lemma}

\begin{proof}
 Consider the simplicial stratification on $BC$. It is a cell
 decomposition of $BC$ indexed by $\overline{N}(C)$, when the topology
 of $C$ is discrete. Even when the topology of $C$ is not discrete, we
 call a stratum of the simplicial stratification a cell in $BC$.
 
 By definition, the image of $n_{x}$ is the union of cells whose
 boundary contains $x$ as a vertex.
 On the other hand, the cells in 
 $D_{x}\times C^{\circ}(\partial D_{x}^{\op})$ are in bijective
 correspondence with
 \[
 P(D_{x}) \cup P(D_{x})\times P(\partial D^{\op}_{x}) \cup P(\partial
 D^{\op}_{x}) = \nN(C\downarrow x) \cup \nN(C\downarrow x)\times
 \nN(C_{>x}) \cup \nN(C_{>x}).
 \]
 The set on the right hand side is the set of nondegenerate chains in $C$
 which contains $x$ or factors through $x$.
 
 Since $n_x$ maps a simplex to a simplex homeomorphically, $n_{x}$ is a
 bijection onto the stratified subspace of $BC$ consisting of cells
 which contain $x$ as a vertex.
 By assumption, $D_{x}\star\partial D_{x}^{\op}$ is compact and $BC$ is
 Hausdorff. And $n_{x}$ is an embedding onto its image.
\end{proof}

\begin{lemma}
 \label{restriction_of_c}
 The restrictions of $c_{x}$ to $D_{x}\times \{1_{x}\}$ and
 $\{1_{x}\}\times D_{x}^{\op}$ coincides with $s_{x}$ and $t_{x}$,
 respectively.
\end{lemma}

\begin{proof}
 It suffices to show that the restrictions of $n_{x}$ to $D_{x}$ and
 $\{1_{x}\}\star \partial D_{x}^{\op}$ can be identified with $s_{x}$
 and $t_{x}$, respectively.

 By the very definition, $n_{x}$ agrees with $s_{x}$ when
 restricted to $D_{x}$.
 It remains to verify the commutativity of the following diagram
 \[
 \xymatrix{
 D_{x}\star \partial D_{x}^{\op} \ar[r]^(0.6){n_{x}} & BC \\
 \{1_{x}\}\star\partial D_{x}^{\op} \ar[u] &
 D_{x}^{\op}. \ar[l]^(0.3){h_{x}^{\op}} 
 \ar[u]_{t_{x}} 
 }
 \]
 For $[\bm{v},\bm{t}]\in\partial D_{x}^{\op}$, we have
 \[
 n_{x}((1-r)1_{x}+r[\bm{v},\bm{t}]) =
 \begin{cases}
  [(v_{0}\circ 1_{x},v_1,\ldots,v_{q}),[(1-r)1+r\bm{t}]], & r>0 \\
  [(1_{x},1)], & r=0,
 \end{cases}
 \]
 which agrees with $t_{x}$ under $h_{x}^{\op}$.
\end{proof}

\begin{corollary}
 The restriction of $c_{x}$ to
 $D_{x}^{\circ}\times D_{x}^{\op,\circ}$ is an open embedding.
 Hence $BC$ is conically stratified.
\end{corollary}

\begin{proof}
 The image of
 $D_{x}^{\circ}\times D_{x}^{\op,\circ}$ under
 $j_{x}\circ(1_{D_{x}}\times h_{x})$ is
 \[
 D_{x}^{\circ}\star \partial D_{x}^{\op} \setminus \partial D_{x}^{\op}
 \cong D_{x}^{\circ}\times \cone^{\circ}(\partial D_{x}^{\op})
 \]
 and restriction of $j_{x}\circ(1_{D_{x}}\times h_{x})$ to
 $D_{x}^{\circ}\times D_{x}^{\op,\circ}$ is a homeomorphism onto its
 image.
 These neighborhoods cover $BC$ and hence $BC$ is conically stratified.
\end{proof}

By Proposition \ref{stratification_on_BC(X)}, we may replace $X$ by
$BC(X)$ when $X$ is a CW stellar complex.

\begin{corollary}
 \label{corollary:exit-path_for_stellar_complex}
 For any cylindrically normal CW stellar complex $X$, $\Exit(X)$ is a
 quasi-category. 
\end{corollary}

%
%

\subsection{Discrete Morse Theory}
\label{discrete_Morse_theory}

Robin Forman \cite{Forman95,Forman98-2} formulated an analogue of Morse
theory for regular cell complexes. For a discrete Morse function
$f:F(X)\to\R$ on the face poset of a finite regular cell complex $X$,
Forman constructed a CW complex $X_{f}$ whose cells are indexed by
critical cells of $f$ and showed that $X_{f}$ is homotopy equivalent to $X$.

Although Forman's discrete Morse theory has been shown to be useful, the
construction of $X_{f}$ is ad hoc. An explicit and functorial
construction would be more useful.
Such a construction was proposed in a joint work of the first author
with Vidit Nanda, Kohei Tanaka \cite{1612.08429}, in which a
poset-enriched category $C(f)$ was constructed from a discrete Morse
function $f:F(X)\to \R$.

For critical cells $c$ and $d$, the set of morphisms $C(f)(c,d)$ has a
structure of poset. By taking the classifying space $B(C(f)(c,d))$ of
each morphism poset, we obtain an acyclic topological category
$BC(f)$ whose set of objects is $C(f)_{0}$.
 
\begin{theorem}[\cite{1612.08429}]
 For a discrete Morse function $f$ on a regular CW complex $X$,
 The classifying space $B^2C(f)=B(BC(f))$ is homotopy equivalent to $X$.
\end{theorem}

As a version of Morse theory, we would like to have a ``cell
decomposition'' of $B^2C(f)$ whose cells are in one-to-one
correspondence with critical cells of $f$.
Theorem \ref{theorem:stratification} and Proposition
\ref{stratification_on_BC(X)} tell us that the 
correct way of decomposing $B^2C(f)$ is a stellar stratification, not
a cell decomposition.

Theorem \ref{theorem:discrete_Morse} can be proved by using the unstable
stratification on $B^2C(f)$.

\begin{proof}[Proof of Theorem \ref{theorem:discrete_Morse}]
 Let $X$ be a finite CW complex. Given a discrete Morse function $f$ on
 $X$, the topological category $BC(f)$ is an acyclic top-enriched
 category. The finiteness of $X$ and a result of de Seguins Pazzis
 \cite{1005.2666} guarantee that the category $BC(f)$ satisfies the
 conditions of Theorem \ref{theorem:stratification}.
\end{proof}

\section{Concluding Remarks}
\label{remarks}

\begin{itemize}
 \item In \cite{1609.04500}, it is proved that if a CNSSS $X$ has a
       ``polyhedral structure'', $BC(X)$ is a strong deformation retract
       of $X$.
       It is very likely that the deformation retraction can be used to
       extend Corollary \ref{corollary:exit-path_for_stellar_complex} to CW
       polyhedral stellar stratified spaces.

 \item The original motivation of this paper was to study relations
       between $C(X)$ and $\Exit(X)$ for a stellar stratified space $X$.
       We anticipate that $C(X)$ and $\Exit(X)$ are equivalent as
       $\infty$-categories if $X$ is a cylindrically normal CW stellar
       complex. This problem will be studied in the sequel to this
       paper.
       
 \item Besides the assumptions of Theorem \ref{theorem:stratification},
       an extra finiteness condition is added in Theorem 
       \ref{theorem:conically_stratified}. This condition is introduced
       only for proving the map $n_{x}$ to be a quotient map in Lemma
       \ref{n_is_embedding}.
       Probably this condition is not necessary or can be replaced by a
       weaker condition.
\end{itemize}

\bibliographystyle{halpha}
\bibliography{%
preamble,%
mathA,%
preprintA,%
mathB,%
preprintB,%
mathC,%
preprintC,%
mathD,%
preprintD,%
mathE,%
preprintE,%
mathF,%
preprintF,%
mathG,%
preprintG,%
mathH,%
preprintH,%
mathI,%
preprintI,%
mathJ,%
preprintJ,%
mathK,%
preprintK,%
mathL,%
preprintL,%
mathM,%
preprintM,%
mathN,%
preprintN,%
mathO,%
preprintO,%
mathP,%
preprintP,%
mathQ,%
preprintQ,%
mathR,%
preprintR,%
mathS,%
preprintS,%
mathT,%
preprintT,%
mathU,%
preprintU,%
mathV,%
preprintV,%
mathW,%
preprintW,%
mathX,%
preprintX,%
mathY,%
preprintY,%
mathZ,%
preprintZ,%
personal}

\def\cprime{$'$} \def\polhk#1{\setbox0=\hbox{#1}{\ooalign{\hidewidth
  \lower1.5ex\hbox{`}\hidewidth\crcr\unhbox0}}}
  \def\cftil#1{\ifmmode\setbox7\hbox{$\accent"5E#1$}\else
  \setbox7\hbox{\accent"5E#1}\penalty 10000\relax\fi\raise 1\ht7
  \hbox{\lower1.15ex\hbox to 1\wd7{\hss\accent"7E\hss}}\penalty 10000
  \hskip-1\wd7\penalty 10000\box7}
  \def\ocirc#1{\ifmmode\setbox0=\hbox{$#1$}\dimen0=\ht0 \advance\dimen0
  by1pt\rlap{\hbox to\wd0{\hss\raise\dimen0
  \hbox{\hskip.2em$\scriptscriptstyle\circ$}\hss}}#1\else {\accent"17 #1}\fi}
  \def\Dbar{\leavevmode\lower.6ex\hbox to 0pt{\hskip-.23ex \accent"16\hss}D}
  \def\cftil#1{\ifmmode\setbox7\hbox{$\accent"5E#1$}\else
  \setbox7\hbox{\accent"5E#1}\penalty 10000\relax\fi\raise 1\ht7
  \hbox{\lower1.15ex\hbox to 1\wd7{\hss\accent"7E\hss}}\penalty 10000
  \hskip-1\wd7\penalty 10000\box7}
  \def\cfudot#1{\ifmmode\setbox7\hbox{$\accent"5E#1$}\else
  \setbox7\hbox{\accent"5E#1}\penalty 10000\relax\fi\raise 1\ht7
  \hbox{\raise.1ex\hbox to 1\wd7{\hss.\hss}}\penalty 10000 \hskip-1\wd7\penalty
  10000\box7}
\begin{thebibliography}{AFT17b}

\bibitem[AFT17a]{1409.0848}
David Ayala, John Francis, and Hiro~Lee Tanaka.
\newblock Factorization homology of stratified spaces.
\newblock {\em Selecta Math. (N.S.)}, 23(1):293--362, 2017,
  \href{http://arxiv.org/abs/1409.0848}{arXiv:1409.0848}.

\bibitem[AFT17b]{1409.0501}
David Ayala, John Francis, and Hiro~Lee Tanaka.
\newblock Local structures on stratified spaces.
\newblock {\em Adv. Math.}, 307:903--1028, 2017,
  \href{http://arxiv.org/abs/1409.0501}{arXiv:1409.0501}.

\bibitem[And10]{AndradeThesis}
Ricardo Andrade.
\newblock {\em From manifolds to invariants of ${E}_n$-algebras}.
\newblock ProQuest LLC, Ann Arbor, MI, 2010.
\newblock Thesis (Ph.D.)--Massachusetts Institute of Technology.

\bibitem[Bj{\"o}84]{Bjorner1984}
A.~Bj{\"o}rner.
\newblock Posets, regular {CW} complexes and {B}ruhat order.
\newblock {\em European J. Combin.}, 5(1):7--16, 1984,
  \url{https://doi.org/10.1016/S0195-6698(84)80012-8}.

\bibitem[CJS]{Cohen-Jones-SegalMorse}
R.~L. Cohen, J.D.S. Jones, and G.~B. Segal.
\newblock Morse theory and classifying spaces,
  \url{http://math.stanford.edu/~ralph/morse.ps}.
\newblock preprint.

\bibitem[dSP13]{1005.2666}
Cl\'ement de~Seguins~Pazzis.
\newblock The geometric realization of a simplicial {H}ausdorff space is
  {H}ausdorff.
\newblock {\em Topology Appl.}, 160(13):1621--1632, 2013,
  \href{http://arxiv.org/abs/1005.2666}{arXiv:1005.2666}.

\bibitem[ERW]{1705.03774}
Johannes Ebert and Oscar Randal-Williams.
\newblock {Semi-simplicial spaces},
  \href{http://arxiv.org/abs/1705.03774}{arXiv:1705.03774}.

\bibitem[For95]{Forman95}
Robin Forman.
\newblock A discrete {M}orse theory for cell complexes.
\newblock In {\em Geometry, topology, \& physics}, Conf. Proc. Lecture Notes
  Geom. Topology, IV, pages 112--125. Int. Press, Cambridge, MA, 1995.

\bibitem[For98]{Forman98-2}
Robin Forman.
\newblock Morse theory for cell complexes.
\newblock {\em Adv. Math.}, 134(1):90--145, 1998,
  \url{http://dx.doi.org/10.1006/aima.1997.1650}.

\bibitem[KJ12]{1009.4227}
Alexander Kirillov~Jr.
\newblock On piecewise linear cell decompositions.
\newblock {\em Algebr. Geom. Topol.}, 12(1):95--108, 2012,
  \href{http://arxiv.org/abs/1009.4227}{arXiv:1009.4227}.

\bibitem[Lur]{LurieHigherAlgebra}
Jacob Lurie.
\newblock {Higher Algebra},
  \url{http://www.math.harvard.edu/~lurie/papers/HA.pdf}.

\bibitem[Lur09]{LurieHigherToposTheory}
Jacob Lurie.
\newblock {\em Higher topos theory}, volume 170 of {\em Annals of Mathematics
  Studies}.
\newblock Princeton University Press, Princeton, NJ, 2009,
  \url{http://dx.doi.org/10.1515/9781400830558}.

\bibitem[LW69]{Lundell-Weingram}
Albert~T. Lundell and Stephen Weingram.
\newblock {\em Topology of CW-Complexes}.
\newblock Van Nostrand Reinhold, New York, 1969.

\bibitem[Mat70]{MatherNotes}
John Mather.
\newblock {Notes on Topological Stability}, July 1970.
\newblock Harvard University.

\bibitem[NTT]{1612.08429}
Vidit Nanda, Dai Tamaki, and Kohei Tanaka.
\newblock {Discrete Morse theory and classifying spaces},
  \href{http://arxiv.org/abs/1612.08429}{arXiv:1612.08429}.

\bibitem[Tam18]{1609.04500}
Dai Tamaki.
\newblock {Cellular stratified spaces}.
\newblock In {\em {Combinatorial and Toric Homotopy: Introductory Lectures}},
  volume~35 of {\em Lecture Notes Series, Institute for Mathematical Sciences,
  National University of Singapre}, pages 305--453. World Scientific
  Publishing, 2018, \href{http://arxiv.org/abs/1609.04500}{arXiv:1609.04500}.

\bibitem[Tho69]{Thom69}
R.~Thom.
\newblock Ensembles et morphismes stratifi\'es.
\newblock {\em Bull. Amer. Math. Soc.}, 75:240--284, 1969,
  \url{https://doi.org/10.1090/S0002-9904-1969-12138-5}.

\end{thebibliography}

\end{document}